\documentclass[a4paper,10pt,reqno]{amsart}
\usepackage{amsmath, amssymb, amsthm}
\usepackage{mathrsfs,enumerate}
\allowdisplaybreaks

\numberwithin{equation}{section}
\newcommand{\hs}[1]{\hskip -#1pt}
\newcommand{\R}{\mathbb{R}}
\newcommand{\1}{ 1 \hspace{-2.4pt} \mathrm{l}}
\newcommand{\N}{\mathbb{N}}

\newcommand{\eps}{\varepsilon}

\newcommand{\loc}{\mathrm{loc}}

\newcommand{\A}{\mathscr{A}}
\newcommand{\G}{\mathscr{G}}
\newcommand{\D}{\mathscr{D}}
\newcommand{\cL}{\mathscr{L}}

\newcommand{\tr}{\mathrm{Tr}}
\newcommand{\intrn}{\int_{\mathbb R^{1+N}}}

\renewcommand{\=}{&\hs{5}=\hs{5}&}

\newcommand{\+}{&\hs{5}&+}
\newcommand{\nno}{\nonumber\\}

\newcommand{\Div}{\operatorname{div}}

\newtheorem{thm}{Theorem}[section]
\newtheorem{prop}[thm]{Proposition}
\newtheorem{cor}[thm]{Corollary}
\newtheorem{lem}[thm]{Lemma}
\theoremstyle{definition}

\newtheorem{hyp}[thm]{Hypothesis}
\theoremstyle{remark}
\newtheorem{rem}[thm]{Remark}

\allowdisplaybreaks
\begin{document}

\title[$L^p$-regularity for parabolic operators]{
$L^p$--regularity for parabolic operators with unbounded
time--dependent coefficients}

\author[M. Geissert]{Matthias Geissert}
\address{M. Geissert,
Technische Universit\"at Darmstadt, Fachbereich Mathematik,
Schlossgartenstr.~7, 64289 Darmstadt, Germany.}
\email{geissert@mathematik.tu-darmstadt.de}
\author[L. Lorenzi]{Luca Lorenzi}
\address{L. Lorenzi, Dipartimento di Matematica, Universit\`a degli Studi
di Parma, Viale G.P. Usberti, 53/A, I-43100 Parma, Italy.}
\thanks{L.L. wishes to thank the Department of Mathematics of the University of
Karlsruhe, where this paper originated, for the kind hospitality during his
visit.}
\email{luca.lorenzi@unipr.it}
\author[R. Schnaubelt]{Roland Schnaubelt}
\address{R. Schnaubelt, Universit\"at Karlsruhe, Fakult\"at f\"ur Mathematik,
 76128 Karlsruhe, Germany.}
\email{schnaubelt@math.uni-karlsruhe.de }



\begin{abstract}
We establish the maximal regularity for nonautonomous Ornstein-Uhlenbeck
operators
in $L^p$-spaces with respect to a family
of invariant measures, where $p\in (1,+\infty)$. This result
follows from the maximal $L^p$-regularity for a class of elliptic operators
 with unbounded, time-dependent drift coefficients and potentials
acting on $L^p(\R^N)$ with Lebesgue measure.
\end{abstract}

\maketitle

\section{Introduction}
In recent years  parabolic problems with unbounded time-independent coefficients
have been investigated intensively. This line of research has focused
on the qualitative behavior, namely on the regularity of solutions
and the properties of invariant measures. (See e.g.\ \cite{lm, BKR06, cerrai,
rockner,stannat99} and the references therein.)
Such parabolic problems arise as Kolmogorov
equations for  ordinary stochastic differential equations. In this context,
however, it is natural to consider time-varying coefficients. Recently a
corresponding analytical theory for nonautonomous Kolmogorov equations
was initiated in \cite{dpl} (see also \cite{daprato}). There and in the papers
\cite{gl1,gl2} the prototypical case of the  nonautonomous
Ornstein-Uhlenbeck operator
\begin{align*}
({\mathscr A}_O(s)\varphi)(x)
=  \frac{1}{2}\sum_{i,j=1}^N  q_{ij}(s)D_{ij}\varphi (x)
   - \sum_{i,j=1}^N b_{ij}(s)x_jD_i\varphi(x), \qquad x\in\R^N,
\end{align*}
was studied, assuming that the coefficients $q_{ij}$ and $b_{ij}$
are bounded and continuous in $s\in\R$ and that the matrix
$[q_{ij}]$ is symmetric and uniformly positive definite.
In this case an explicit
formula for the solution of the parabolic equation
\begin{equation} \label{pb-3}
\left\{
\begin{array}{ll}
D_su(s,x)={\mathscr A}_O(s)u(s,x), & s\ge r, \ x\in\R^N,\\[2mm]
u(r,x)=\varphi (x), & x\in\R^N,
\end{array}
\right.
\end{equation}
when $\varphi\in C_b(\R^N)$ is known. This formula is very useful in many respects
(e.g.,\ to study regularity), see
\cite{dpl,gl1,gl2}; but it will play no role in our
investigations. The solutions of \eqref{pb-3} define evolution operators
(or, an evolution family)
on $C_b(\R^N)$ by setting $G_O(s,r)\varphi:=u(s)$. Recently, the
results from \cite{dpl,gl1,gl2} have partly been
extended to more  general elliptic operators with  time-varying unbounded
coefficients, see \cite{KLL,LZ}.

Under suitable assumptions, autonomous Kolmogorov operators admit an
invariant measure.
As the results in \cite{gl1} show, this is not true anymore in the nonautonomous case,
which is in fact the crucial novelty in the case of time-varying coefficients.
However, in \cite{gl1} it has been proved that it is possible to obtain
a \emph{family of invariant measures}
$\{\nu_s, \ s\in\R\}$ (also called \emph{evolution system of invariant
measures} in \cite{daprato} and \emph{entrance laws at $-\infty$} in
\cite{dynkin}), provided the matrices $-[b_{ij}(s)]$
generate an exponentially stable evolution family on $\R^N$.
These measures  are Borel probability measures on $\R^N$
satisfying the equation
\begin{align}\label{inv-def}
  \int_{\R^N} G_O(s,r)\varphi\,d\nu_s =\int_{\R^N}\varphi\,d\nu_r
\end{align}
for all $\varphi\in C_b(\R^N)$ and all $r,s\in\R$ with $r\le s$.
The set of all such families of invariant measures
has been described in \cite{gl1},
and it was shown that there exists exactly one family $\{\mu_s, \ s\in\R\}$
of Gaussian type which has finite moments of every order. In formula
\eqref{mu} we recall the explicit formula for $\mu_s$.
The existence of families of invariant measures for more general nonautonomous
operators has recently been proved in \cite{KLL}, see also \cite{BDPR08,BRS06,BRS08}
for related results.

The defining property \eqref{inv-def} of invariant measures easily implies
that one can extend the evolution operator associated with \eqref{pb-3} to a contraction
$G_O(s,r):L^p(\R^N,\mu_r) \to L^p(\R^N,\mu_s)$ for
all $s\ge r$. As in the autonomous
case one can expect good regularity properties of this extension. But
in the nonautonomous case one has to pay the price that the evolution operators
act on a family of spaces. In addition, it is well known that the asymptotic
behavior of nonautonomous problems is much more difficult to treat than in the
autonomous case. For an evolution family on a fixed Banach space an associated
 `evolution semigroup' was introduced for the study of evolution families.
 For instance, this semigroup allows to derive spectral theoretic
characterizations of certain asymptotic properties of the evolution family,
 see \cite{ChiLa}, \cite{Sch-survey}. It was observed by Da Prato and Lunardi in \cite{dpl}
that  one can generalize this construction also to the case of $L^p$-spaces
with time-varying measures, and the authors used  the evolution semigroup
in the study of longterm behavior of $G_O$,
see also \cite{gl1,gl2,KLL}.

Following these papers, we define a measure
$\nu$ on Borel sets on $\R^{1+N}$ by setting
\begin{align}\label{nu}
\nu(J\times B)=\int_J\mu_s(B)\,ds
\end{align}
for Borel sets $B\subset \R^N$ and $J\subset \R$. Of course, $\nu$ is
not a probability measure anymore. One further introduces the
\emph{evolution semigroup} $T(\cdot)$ on $L^p(\R^{1+N},\nu)$  corresponding to
$G_O$ defined by
\begin{equation}
(T(t)f)(s,x)=(G_O(s,s-t)f(s-t,\cdot))(x),\qquad\;\,(s,x)\in\R^{1+N},\;\,t\ge 0,
\label{sem-T(t)}
\end{equation}
where $f\in L^p(\R^{1+N},\nu)$ and $1\le p<+\infty$. It is straightforward to
check that equation \eqref{sem-T(t)} defines in fact a $C_0$-semigroup on
$L^p(\R^{1+N},\nu)$ and that
\begin{equation*}
\int_{\R^{1+N}}T(t)f\,d\nu=\int_{\R^{1+N}}f\, d\nu,
    \qquad t>0,\;\,f\in C_c(\R^{1+N}),
\end{equation*}
see \cite{dpl} or \cite{gl1}.
We denote the generator of $T(\cdot)$ in $L^p(\R^{1+N},\nu)$ by $G_p$,
where $1\le p<+\infty$. In \cite{LZ} it has been proved that $G_p$
is the closure  the parabolic operator ${\mathscr G}$ defined by
\begin{align*}
({\mathscr G}u)(s,x)
   =({\mathscr A}_O(s)u(s,\cdot))(x)-D_s u(s,x),\qquad\;\,(s,x)\in\R^{1+N},
\end{align*}
for  $u\in C^{\infty}_c(\R^{1+N})$. In this paper we want to show that
$G_p$ has the `natural' domain
\begin{align}\label{dom-g}
D(G_p)&=\{u\in L^p(\R^{1+N},\nu): D_t,D_i u, D_{ij}u\in L^p(\R^{1+N},\nu),
\ \forall\; i,j=1,\ldots,N\}\notag\\
&=:W^{1,2}_p(\R^{1+N},\nu),
\end{align}
for  $1< p<+\infty$,
see Theorem~\ref{ou-thm}. This means that for each inhomogeneity
$f\in L^p(\R^{1+N},\nu)$ and each $\lambda>0$, the function $u=(\lambda-G_p)^{-1}f$
is the only solution in $W^{1,2}_p(\R^{1+N},\nu)$ of the parabolic equation
 \begin{align}\label{ou-line}
 D_s u(s)= (\A_O(s)-\lambda)u(s) + f(s), \qquad s\in\R,
 \end{align}
 on the line.
In other words, the problem \eqref{ou-line} possesses maximal $L^p$-regularity
with respect to the measure $\nu$. Such results are known in the
autonomous case even in much greater generality, see e.g.\
\cite{CG01,DV02,Lu97,MN,MPRS0,MPRS} and the references therein, where a variety
of methods was developed. In the case $p=2$, the identity \eqref{dom-g} was
shown in \cite{gl1} for the nonautonomous case using regularity
properties of $G_O(s,r)$ and tools from interpolation theory. However,  the necessary
results from interpolation theory do not hold for $p\neq 2$.

In this paper
we establish \eqref{dom-g} for all $p\in (1,+\infty)$ using a completely
different method, inspired by \cite{DV02} and \cite{MPRS}. We transform the
operator $\G$ into an operator $\cL_O$ on
the space $L^p(\R^{1+N})$ with Lebesgue measure which has a dominating
potential, see Section~\ref{sect-2}. The operator $\cL_O$ is a (simple)
special case of a class of parabolic operators  $\cL=\A(\cdot)-D_s$
on $L^p(\R^{1+N})$ with time-varying coefficients, see \eqref{scrA}.
The uniformly elliptic
operators $\A(s)$ may have unbounded potential and
drift coefficients. We require
that the potential  satisfies an oscillation condition and that it dominates
the drift coefficients, as described in Section~\ref{sect-3}.
In Theorem~\ref{Lp-thm}
it is shown that the realization $L_p$ of $\cL$ in $L^p(\R^{1+N})$,
where  $1<p<+\infty$, with the domain
$D(L_p)=W^{1,2}_p(\R^{1+N})\cap D(V)=:\D_p$  generates a positive and
contractive evolution semigroup $S(\cdot)$. Hence the parabolic equation
\begin{align}\label{parab0}
 D_s u(s)= (\A(s)-\lambda)u(s) + f(s), \qquad s\in\R,
 \end{align}
has the unique solution $u=(\lambda-L_p)^{-1}f$ in $\D_p$,
 for every $f\in L^p(\R^{1+N})$ and $\lambda>0$;
i.e., \eqref{parab0} has maximal $L^p$-regularity.
 Moreover, the evolution family associated with $S(\cdot)$
solves the initial value problem corresponding to \eqref{parab0}. In
Section~\ref{sect-4} we extend these results to the spaces $L^1(\R^{1+N})$
and $C_0(\R^{1+N})$.

By means of Theorem~\ref{Lp-thm} one could also treat generalized
Ornstein-Uhlenbeck operators as in \cite{DV02,MPRS}. For simplicity, we
restrict ourselves to the basic and most prominent case of the classical
Ornstein-Uhlenbeck operators.

Our main theorems are based on two crucial estimates and on
semigroup theory. In
Proposition~\ref{interpol1} we show a weighted gradient estimate
which allows to control the gradient term by the heat operator and
the potential. Proposition~\ref{prop-dom-p} gives the main
\emph{a priori} estimate for the parabolic operator ${\mathscr L}$
which implies that its realization $L_p$ with domain $\D_p$ is closed in
$L^p(\R^{1+N})$. We then verify that $L_p$ is maximally dissipative
and employ the theory of evolution semigroups to establish
Theorem~\ref{Lp-thm}. The proofs for the spaces $L^1$ and $C_0$
in the fourth section are similar, and the one for $C_0$ uses the
$L^p$ result. Our approach
is inspired by the paper \cite{MPRS} which was devoted to the
autonomous case, but there are fundamental differences. So we cannot use  the
theory of analytic semigroups since the evolution semigroup $S(\cdot)$
is not analytic. (The spectrum of its generator contains
vertical lines, see \cite[Theorem~3.13]{ChiLa}.) Further, the known
results on parabolic evolution operators do not apply to the class
of elliptic operators $\A(s)$ studied here, see Remark~\ref{AT-rem}.
 Moreover, the presence of the
time derivative in $\cL$ leads to new difficulties in the proofs
of Propositions~\ref{interpol1} and \ref{prop-dom-p}.
For instance, we need a parabolic version of the Besicovitch covering theorem
established in the Appendix.

Besides \cite{MPRS} and the papers mentioned above, there are several
works treating $L^p$-regularity for autonomous problems
with unbounded coefficients in $L^p$-spaces with respect to the Lebesgue measure,
see e.g.\ \cite{CV87,CV88,FL,PRS} and the references therein.
We are only  aware of one related paper for nonautonomous problems
(except for \cite{gl2}): in \cite{CMS}
operators without drift terms were studied with completely different
methods and assumptions.

\subsection*{Notations}
We denote by $|\cdot|$ the Euclidean norm of vectors,
whereas $\|A\|$ is the operator  norm of a matrix with respect to
the Euclidean norm. The transpose of $A$ is $A^*$ and $\tr A$ is its trace.
Open balls in $\R^d$ are designated by $B(x,r)$.
We write $\langle\xi, \eta\rangle$ or $\xi\cdot\eta$ for the
 scalar product in $\R^d$ and $I$ for the identity map.
   $D_j$, $\nabla$, $D^2$ and $\Div$ are the (distributional)
 partial derivatives, gradient, Hessian matrix and divergence, respectively,
  with respect to the space variable $x\in \R^N$. We also use the notations
  $\nabla_x$, $D_x^2$ and $\Div_x$  if a function depends
 on both the time and space variables $(s,x)\in \R^{1+N}$. In this case $D_s$
 is the time derivative. We always denote  the spatial Laplace operator
 by $\Delta=D_1^2+\dots +D_N^2$.

In this paper we only consider real function spaces. The symbol $C^k$
refers to spaces of $k$-times continuously differentiable functions,
where $k\in\N\cup\{0,+\infty\}$. In such spaces
the subscript $c$ means  `with compact support', whereas the subscript
$b$ (resp. $0$)
means that the functions and the derivatives up to order $k$ are bounded
(resp. vanish at $\infty$).
The space of continuous functions $f:\R^{1+N}\to \R^d$ such that also
$\nabla_x f$ is continuous on $\R^{1+N}$ is denoted by
$C^{0,1}(\R^{1+N},\R^d)$.
Let $\mu$ be a $\sigma$-finite measure on $\R^{1+N}.$
Then, $W^{1,2}_p(\R^{1+N},\mu)$ is
the space of functions $f:\R^{1+N}\to\R$ such that $f$ and
the distributional derivatives $D_sf$,
$D_jf$ and $D_{ij}f$ ($i,j=1,\ldots,N$) belong to $L^p(\R^{1+N},\mu)$.
We endow it with the natural norm
\begin{eqnarray*}
\|f\|_{W^{1,2}_p(\R^{1+N},\mu)}^p\=\int_{\R^{1+N}}|f|^pd\mu+
\intrn |D_sf|^pd\mu\nno
\+\sum_{j=1}^N\int_{\R^{1+N}}|D_jf|^pd\mu+
\sum_{i,j=1}^N\int_{\R^{1+N}}|D_{ij}f|^pd\mu.
\end{eqnarray*}
We use analogous definitions for subsets of the form $(a,b)\times \R^N$.
If $\mu$ is the Lebesgue measure, we omit $\mu$ in the notation.
The usual isotropic Sobolev spaces on $\R^d$ are denoted by
$W^k_p(\R^d)$. The norm on $L^p(\R^d)$ is designated by $\|f\|_p$
for $1\le p\le +\infty$.
Finally, we write $c=c(\alpha,\ldots)$ for a constant depending only
on the quantities $\alpha,\ldots$ Such constants may vary from line to line.


\section{Transformation of the parabolic Ornstein--Uhlenbeck operator}
\label{sect-2}

For any continuous function $s\mapsto B(s)$ from $\R$ into the set of $N\times N$ matrices, we denote by $U(s,r)$
the solution of the problem
\begin{align*}
\left\{
\begin{array}{ll}
D_sU(s,r)=B(s)U(s,r), & s\in\R, \\[2mm]
U(r,r)=I,
\end{array}
\right.
\end{align*}
where $r\in\R$. We state our hypotheses on
the coefficients $Q(s) =[q_{ij}(s)]$ and $B(s) = [b_{ij}(s)]$ of the
Ornstein-Uhlenbeck operator $\A_O(s)$.
\begin{hyp}\label{hyp1}
\begin{itemize}
\item[(i)]
The coefficients $q_{ij}$ and $b_i$  belong, respectively, to  $C_b^1(\R)$ and
$C_b(\R)$ for all $i,j=1,\ldots,N$.
\item[(ii)] For every $s\in\R$, the matrix $Q(s)$ is symmetric and there
exists a constant $\eta_0>0$ such that
\begin{eqnarray*}
\langle Q(s)\xi , \xi \rangle \geq \eta_0 |\xi |^2, \qquad\;\,
\xi \in \R^N,\ s\in\R.
\end{eqnarray*}
\item[(iii)]
There exist constants $C_0,\omega>0$ such that
\begin{eqnarray*}
\|U(r,s)\|\le C_0e^{-\omega(s-r)},\qquad\;\,
s,r\in\R \text{ with } s\ge r.
\end{eqnarray*}
\end{itemize}
\end{hyp}

Under the above assumptions, there exists a family of invariant measures
for $\A_O(s)$ (see \eqref{inv-def}) of Gaussian type  given by
\begin{align} \label{mu}
\mu_s(dx)&= (2\pi)^{-\frac{N}{2}}(\det Q_s)^{- \frac12}
 \,e^{-\frac{1}{2}\langle Q_s^{-1}x,x\rangle},\qquad s\in\R,\ x\in\R^N,\\
Q_s &=\int_s^{+\infty}U(s,\xi)Q(\xi)U^*(s,\xi)d\xi,\qquad s\in\R, \label{qs}
\end{align}
see \cite{dpl} and \cite{gl1}.
Actually, the authors of the previous papers deal with backward nonautonomous parabolic problems,
whereas we have preferred to consider forward problems in the present paper.
But a straightforward change of variables allows to transform
the problem \eqref{pb-3} into
a backward Cauchy problem. More precisely, for any $r\in\R$,
the function $(s,x)\mapsto v(s,x):=(G_O(-s,-r)\varphi)(x)$
is a classical solution to the backward Cauchy problem
\begin{equation}\label{pb-3-bis}
\left\{
\begin{array}{ll}
D_sv(s,x)+\hat {\mathscr A}_O(s)v(s,x)=0, & s\le r, \ x\in\R^N,\\[2mm]
v(r,x)=\varphi(x), & x\in\R^N,
\end{array}
\right.
\end{equation}
where
\begin{eqnarray*}
\hat{\mathscr A}_O(s)\varphi &\hs{5}=\hs{5}& \frac{1}{2}\sum_{i,j=1}^N
q_{ij}(-s)D_{ij}\varphi - \sum_{i=1}^N b_{ij}(-s)x_jD_i\varphi.
\end{eqnarray*}
Hence, the evolution operator  $G_O(s,r)$ associated with problem
\eqref{pb-3} and the evolution operator  $P(s,r)$ associated with
problem \eqref{pb-3-bis} are related by the formula
\begin{equation*}
G_O(s,r)\varphi=P(-s,-r)\varphi,\qquad\;\,r\le s,\;\,\varphi\in C_b(\R^N).
\end{equation*}

In the first lemma we collect some estimates concerning the
densities of the invariant measures.
\begin{lem} \label{prop-1}
Assume that Hypothesis $\ref{hyp1}$ is satisfied.
Then, there exist two constants $C_1,C_2>0$ such that the inequalities
\begin{align}
C_1|x|^2 &\le \langle Q_rx,x\rangle\le C_2|x|^2,   \label{estim-Qt-0}\\
C_2^{-1}\,|x|&\le | Q_r^{-1}x| \le C_1^{-1}\,|x|,   \label{estim-Qt-2} \\
  C_1^N    & \le \det Q_r\le  C_2^N,  \label{estim-Qt-1}
\end{align}
hold for all $r\in\R$ and $x\in\R^N$.
\end{lem}

\begin{proof}
Let  $x\in\R^N$ and $r\in\R$. Formula \eqref{qs} and
Hypothesis~\ref{hyp1} yield that
\begin{align} \label{qsxx}
\langle Q_rx,x\rangle
&=\int_r^{+\infty}\langle Q(\xi)U^*(r,\xi)x, U(r,\xi)^*x\rangle \, d\xi\\
&\le  C_0^2\|Q\|_{\infty}|x|^2\int_r^{+\infty}e^{-2\omega(\xi-r)}d\xi
    =\frac{C_0^2}{2\omega}\|Q\|_{\infty}\,|x|^2,\notag
\end{align}
for any $x\in\R^N$,
which accomplishes the proof of the second inequality in \eqref{estim-Qt-0}
with $C_2=\frac{C_0^2}{2\omega}\|Q\|_{\infty}$. We further recall that
$U(r,s)^{-1}=U(s,r)$  for all $r,s\in\R$ and that
$\|U(r,s)\|\le M_0 e^{\varpi(r-s)}$ for constants $\varpi\in\R_+$
and $M_0\ge1$ and all $r\ge s$.
It thus holds
\[  |x|=|U^*(\xi,r)U^*(r,\xi)x|\le \| U(\xi,r)\| \ |U^*(r,\xi)x|
\le M_0e^{\varpi (\xi-r)}\, |U^*(r,\xi)x|,  \]
for all $r,\xi\in\R$ with $r\le \xi$ and all $x\in\R^N$. Using \eqref{qsxx}
and Hypothesis~\ref{hyp1}(ii), we then deduce
\[
\langle Q_rx,x\rangle\ge\eta_0\int_r^{+\infty}|U^*(r,\xi)x|^2\, d\xi
\ge \frac{\eta_0|x|^2}{M_0^2}\int_r^{+\infty}e^{-2\varpi(\xi-r)}\,d\xi
=\frac{\eta_0}{2M_0^2\varpi}|x|^2,
\]
which gives the first estimate in \eqref{estim-Qt-0} with
$C_1=\eta_0(2M_0^2\varpi)^{-1}$. The assertion \eqref{estim-Qt-0}
is equivalent to
\begin{equation}
\sqrt{C_1}\,|x|\le |Q^{1/2}_rx|\le\sqrt{C_2}\, |x|.  \label{estim-Qt-3}
\end{equation}
The first inequality in \eqref{estim-Qt-2} now follows noting that
\begin{equation*}
|x|= |Q_r^{1/2} Q_r^{1/2} Q^{-1}_rx| \le
\| Q_r^{1/2}\|^2 \; |Q^{-1}_rx|\le C_2\, |Q^{-1}_rx|.  
\end{equation*}
On the other hand, \eqref{estim-Qt-3} implies $\|Q_r^{-1/2}\|\le C_1^{-1/2}$
 and, hence, the second part of \eqref{estim-Qt-2}.
The final assertion \eqref{estim-Qt-1} is a consequence of the fact that
the eigenvalues of $Q_r$ belong to the interval
$[C_1,C_2]$  due to \eqref{estim-Qt-0}.
\end{proof}

Let $p\in (1,+\infty)$.  We now transform the differential operator
$\G ={\mathscr A}_O(\cdot)-D_s$ acting on $L^p(\R^{1+N},\nu)$
into a differential operator $\cL_O$ acting on $L^p(\R^{1+N})$.
To this purpose, we set $\Phi(s,x)=\frac12 \langle Q_s^{-1}x,x\rangle$
for $(s,x)\in \R^{1+N}$. Observe that \eqref{nu}, \eqref{mu} and
\eqref{estim-Qt-1} yield
\begin{align*}
&\int_{\R^{1+N}}\!|e^{\frac{1}{p}\Phi} f|^p\,d\nu
  =(2\pi)^{-\frac{N}{2}}\!\int_{\R^{1+N}}\!(\det Q_s)^{-\frac12} |f|^p\,ds\,dx
   \le (2\pi C_1)^{-\frac{N}{2}}\! \int_{\R^{1+N}}\! |f|^p ds\,dx, \\[2mm]
&\int_{\R^{1+N}}| e^{-\frac{1}{p}\Phi}g|^p\,ds\,dx
  = (2\pi)^{\frac N2} \int_{\R^{1+N}}(\det Q_s)^\frac12 |g|^p\,d\nu
 \le (2\pi C_2)^{\frac N2}\int_{\R^{1+N}} |g|^p \,d\nu,
\end{align*}
for every $f\in L^p(\R^{1+N})$ and  $g\in L^p(\R^{1+N},\nu)$.
Therefore the operator $M_p:L^p(\R^{1+N})\to L^p(\R^{1+N},\nu)$, defined by
\begin{equation}\label{Mp}
(M_pf)(s,x)=e^{\frac{1}{2p}\langle Q_s^{-1}x,x\rangle}f(s,x)
           =e^{\frac{1}{p}\Phi(s,x)}f(s,x),
\end{equation}
is an isomorphism with the inverse $M_p^{-1} g= e^{-\frac{1}{p}\Phi }g$.
On test functions we now define the differential operator
\begin{equation}\label{LO}
{\mathscr L}_O  :=M_p^{-1}({\mathscr A}_O(\cdot)-D_s)M_p.
\end{equation}
Let $u$ be smooth.
A straightforward computation shows that the equalities
\begin{equation} \label{change-unknown}
\begin{split}
D_sM_p u &= \frac{1}{p} (M_p u) D_s\Phi  +M_p(D_s u),\\
D_iM_p u &= \frac{1}{p}  e^{\frac{1}{p}\Phi} u D_i\Phi +e^{\frac{1}{p}\Phi}D_i u
          = \frac{1}{p} (M_p u) D_i\Phi  +M_p(D_i u),\\
D_{ij}M_p u &= \frac{1}{p} (M_pu) D_{ij}\Phi
       +\frac{1}{p^2}  (M_p u) (D_i\Phi) D_j\Phi
       +\frac{1}{p} (D_i\Phi) M_p(D_j u)\\
       &\qquad   +\frac{1}{p} (D_j\Phi) M_p(D_i u) +M_p(D_{ij} u)
\end{split}
\end{equation}
hold on $\R^{1+N}$ for all $i,j=1,\ldots,N$. For  $(s,x)\in\R^{1+N}$
 we thus obtain
\begin{align*}
({\mathscr L_O}u)(s,x)&= -D_s u(s,x)+\frac{1}{2}\tr(Q(s)D^2_x u(s,x))
       -\langle B(s)x,\nabla_x u(s,x)\rangle\\
&\qquad + \frac{1}{p}\langle Q(s)\nabla_x\Phi(s,x),\nabla_x u(s,x)\rangle
     +\frac{1}{2p}\tr(Q(s)D^2_x\Phi(s,x))u(s,x)\\
&\qquad +\frac{1}{2p^2}\langle Q(s)\nabla_x\Phi(s,x),\nabla_x\Phi(s,x)\rangle
 u(s,x) -\frac{1}{p}u(s,x)D_s\Phi(s,x)\\
&\qquad -\frac{1}{p}\langle B(s)x,\nabla_x\Phi(s,x)\rangle u(s,x)\\
&\hspace*{-1.1cm}=:-D_su(s,x) +\frac{1}{2}\tr \left(Q(s) D^2_xu (x)\right) +
   \langle F_O(s,x), \nabla_x u(s,x) \rangle -V_O(s,x)u(s,x).
\end{align*}
To write $F_O$ and $V_O$ more conveniently,  we observe that
\begin{equation}\label{D-Phi}
\nabla_x\Phi(s,x)=Q_s^{-1}x \qquad \text{and} \qquad D^2_x\Phi(s,x)=Q_s^{-1}.
\end{equation}
As a consequence,
 \begin{align}\label{def-FO}
 F_O(s,x)= \frac{1}{p} Q(s)Q_s^{-1}x -B(s)x.
\end{align}
 We further have
\begin{align}\label{Ds-Phi}
D_s\Phi(s,x)&= \frac{1}{2}\langle D_sQ_s^{-1}x,x\rangle=
       -\frac{1}{2}\langle Q_s^{-1}\,(D_sQ_s)\,Q_s^{-1}x,x\rangle \notag\\
&= -\frac{1}{2}\left\langle Q_s^{-1}\left (D_s\int_s^{+\infty}
     U(s,r)Q(r)U^*(s,r)\,dr\right )Q_s^{-1}x,x\right\rangle   \notag\\
&= -\frac{1}{2}\langle Q_s^{-1}\left (-Q(s)+B(s)Q_s+Q_sB^*(s)\right )
   Q_s^{-1}x,x\rangle\notag\\
&= \frac{1}{2}\left(\langle Q(s)Q_s^{-1}x,Q_s^{-1}x\rangle
    -\langle Q_s^{-1}B(s)x,x\rangle-\langle B^*(s)Q_s^{-1}x,x\rangle\right) \notag \\
&=  \frac{1}{2}\langle Q(s)Q_s^{-1}x,Q_s^{-1}x\rangle
    -\langle Q_s^{-1}B(s)x,x\rangle.
\end{align}
It follows that
\begin{align}\label{def-VO}
V_O(s,x)&=-\frac{1}{2p}\tr(Q(s)Q_s^{-1})-\frac{1}{2p^2}\langle Q(s)Q_s^{-1}x,Q_s^{-1}x\rangle
+\frac{1}{p}\langle B(s)x,Q_s^{-1}x\rangle \notag\\
&\qquad +\frac{1}{p}\left (\frac{1}{2}\langle Q(s)Q_s^{-1}x,Q_s^{-1}x\rangle
-\langle Q_s^{-1}B(s)x,x\rangle\right ) \notag\\
&= \frac{1}{2p}\left (1-\frac{1}{p}\right )\langle Q(s)Q_s^{-1}x,Q_s^{-1}x\rangle
-\frac{1}{2p}\tr (Q(s)Q_s^{-1}),
\end{align}
for all $(s,x)\in\R^{1+N}$. Now, let $p\in(1,+\infty)$.
Hypothesis~\ref{hyp1} and Lemma~\ref{prop-1} then imply that
 $V_O(s,x)\ge k_1\,|x|^2-k_0$
for constants $k_1=k_1(p)>0$ and $k_0\ge0$ and all $(s,x)\in\R^{1+N}$.
We fix the number  $c_0= 2\|\Div_x F_O\|_\infty$
(which is possible because of Hypothesis~\ref{hyp1},
Lemma~\ref{prop-1} and \eqref{def-FO}) and set
\[W_O(s,x)=c_0+k_1\,|x|^2,\]
 for all $(s,x)\in\R^{1+N}$.
In view of Lemma~\ref{prop-1} and formulas  \eqref{def-FO} and
\eqref{def-VO},  there exist constants
 $\lambda=k_0+c_0\ge0$, $c_1=c_1(p)\ge1$, $\kappa=\kappa(p)>0$,
 and $\theta=2/3$  with
\begin{equation}\label{VF-cond}
W_O\le \lambda + V_O \le c_1 W_O, \qquad |F_O|\le \kappa W_O^{1/2},
\qquad \theta W_O+ \Div_x F_O\ge 0
\end{equation}
 on $\R^{1+N}$.

\section{Operators with dominating potential for $1<p<+\infty$.}
\label{sect-3}

In this section we mainly consider elliptic operators of the form
\begin{align}\label{scrA}
\A(s)\varphi=\Div_x(a(s)\nabla_x \varphi)+F(s)\cdot \nabla_x \varphi
      -V(s)\varphi,
\end{align}
at first defined for  $\varphi \in C^\infty_c(\R^N)$, and their
parabolic counterpart
\[ \cL u =(\A(\cdot)-D_s) u,\]
at first defined for $u\in C^\infty_c(\R^{1+N})$.  We assume the
following conditions on the coefficients $a=[a_{ij}]$, $F$ and $V$.
\begin{enumerate}
\item[{\bf (A1)}] $a_{ij}\in C^1_b(\R^{1+N})$ satisfy
$a_{ij}=a_{ji}$ and
$$
\sum_{i,j=1}^Na_{ij}(s,x)\xi_i\xi_j\ge \eta_0 |\xi |^2,
$$
 for all $x,\,\xi \in \R^N$, $s\in\R$, $i,j\in\{1,\ldots,N\}$ and
some constant $\eta_0 >0.$
\item[{\bf (A2)}] $W\in C^1(\R^{1+N})$ is a function such that
$W\ge c_0>0$, $|D_sW| \le \beta W^2 + K_\beta$ and
$|\nabla_x W|\le \gamma W^{\frac{3}{2}}+ K'_\gamma$
for some constants $c_0, \beta,\gamma>0$ and $K_\beta, K'_\gamma \geq 0$.
\item[{\bf (A3)}] $V:\R^{1+N}\to\R$ is measurable and
$ W\le V\le c_1 W$ for some constant $c_1\ge1$.
\item[{\bf (A4)}] $F\in C(\R^{1+N},\R^N)$ satisfies
$|F|\le \kappa W^{\frac{1}{2}}$ for some constant $\kappa \ge 0$.
\item[{\bf (A5)}] $F\in C^{0,1}(\R^{1+N},\R^N)$ and there
exists a constant $\theta \in[0,p)$ such that
$\theta W+{\Div_x}\,F\ge 0,$ where $p\in[1,+\infty)$ is given.
\end{enumerate}

Later on we will impose additional restrictions on the size of $\beta$ and $\gamma$,
see \eqref{small}.
Due to \eqref{VF-cond} the functions $Q$, $F_O$, $\lambda+V_O$ and $W_O$ from
the previous section satisfy (A1)--(A5) with $\beta=K_\beta=0$ and
arbitrarily small $\gamma>0$, for each $p\in(1,+\infty)$. Except for the
estimate on $D_sW$, the hypotheses (A1)--(A5) were already used
 in \cite{MPRS} in the autonomous case. We want to discuss them shortly,
referring the reader to \cite{MPRS} for more details and further references.
Of course, (A1) gives uniform ellipticity. Assumption (A4)
allows us to control the drift term by the potential. (But note that
the drift term is not a small perturbation, cf.\ \cite[Remark~3.6]{MPRS}.)
The inequality in (A5) is a slightly strengthened  dissipativity condition
for $\cL$. The crucial hypothesis is (A2) which restricts the oscillation
of the
auxiliary potential $W$, whereas (A3) allows to compare $V$ and $W$.
The use of $W$ gives some more flexibility in the applications
(as already exploited in \cite[Section 7]{MPRS}). Example~3.7 in \cite{MPRS}
shows that one cannot omit (A2) and that even the restriction in
\eqref{small} is almost sharp in certain cases.

In this section
we want to show that  $\cL$, endowed with  the domain
\[
\D_p:= \{ u\in W^{1,2}_p(\R^{1+N}): Wu \in L^p(\R^{1+N})\},
\]
generates a $C_0$-semigroup on $ L^p(\R^{1+N})$ and we want to exploit
this fact in the study of \eqref{ou-line} and its variant
\eqref{parab-eq} for $\A(\cdot)$. In the next section we also use the domains
\begin{align*}
\D_1&:=\{u\in L^1(\R^{1+N}):  (\Delta-D_s)u, Wu\in L^1(\R^{1+N})\},\\
\D_\infty&:=\{u\in C_0(\R^{1+N}): (\Delta-D_s)u, Wu\in C_0(\R^{1+N})\}\\
   &\,=\{u\in C_0(\R^{1+N})\cap W^{1,2}_{q,\loc}(\R^{1+N}) \;\forall\,q<+\infty:
   (\Delta-D_s)u, Wu\in C_0(\R^{1+N})\},
\end{align*}
where the last equality follows from standard local parabolic regularity.
The spaces $\D_p$, $1\le p\le +\infty$, are endowed with their
natural norms given by
\begin{align*}
\|u\|_{\D_p}^p&= \|u\|_{W^{1,2}_p(\R^{1+N})}^p+\|Wu\|_p^p\,,
           \qquad 1<p<+\infty,\\
\|u\|_1 &= \|(\Delta-D_s)u\|_1 + \|Wu\|_1, \\
  \|u\|_\infty &= \max\{\|(\Delta-D_s)u\|_\infty\,, \|Wu\|_\infty\}.
\end{align*}
Note that in the definitions of the spaces $\D_p$ and their norms, one could replace everywhere
$W$  by $V$ getting the same sets and equivalent norms (where $V$ is assumed
to be continuous if $p=\infty$). We recall that the norm on
$W^{1,2}_p(\R^{1+N})$ is equivalent to the graph norm of $\Delta-D_s$
on  $ L^p(\R^{1+N})$ if $p\in(1,+\infty)$. At first, we prove  three more
or less standard facts  for every $p\in[1,+\infty]$,

\begin{lem}\label{dissip}
Assume that hypothesis {\rm (A1)} is satisfied. Then, the following assertions
hold.
\begin{enumerate}[(a)]
\item
If $F\in C^{0,1}(\R^{1+N},\R^{N})$,  $V\in L^p_{\loc}(\R^{1+N})$,  and
$\;V+\frac{1}{p}\Div_x F\ge 0\;$
for some  $1\le p<+\infty$, then  $(\cL, C_c^\infty(\R^{1+N}))$
is dissipative in $L^p(\R^{1+N})$.
\item
If $F\in C(\R^{1+N},\R^N)$, $V\in C(\R^{1+N})$ and $V\ge0$,
then  $(\cL, C_c^\infty(\R^{1+N}))$ is dissipative in $C_0(\R^{1+N})$.
\end{enumerate}
\end{lem}
\begin{proof}
 Let $1\le p<+\infty$. It is known that
\[
\int_{\R^N} (\A(s)\varphi) \,\varphi|\varphi|^{p-2}\,dx \le 0,
\]
 for all $\varphi\in C_c^\infty(\R^N)$ and $s\in\R$,
see e.g.,\ \cite[Lemma~2.6]{MPRS}. For $u\in C_c^\infty(\R^{1+N})$ we thus
obtain
\[
 \int_{\R^{1+N}} (\A(\cdot) u -D_s u)\,u |u|^{p-2}\,ds\,dx
 \le  -\frac{1}{p} \int_{\R^{1+N}} D_s\, |u|^p\,ds\,dx = 0.
\]
This shows assertion (a).  The dissipativity of $\cL$ in $C_0(\R^{1+N})$
is a  standard consequence of the maximum principle.
\end{proof}

\begin{lem}\label{interpol}
For every $u \in C_c^\infty (\R^{1+N})$ and  $1 \le p\le +\infty$,  we have
\begin{equation*}
\|\nabla_x u\|_p \leq C\,\|(\Delta-D_s) u\|_p ^{\frac{1}{2}}\,
            \|u\|_p^{\frac{1}{2}},
\end{equation*}
with a constant $C>0$ depending only on $N$.
\end{lem}
\begin{proof}
For a given $\lambda>0$ and $u \in C_c^\infty (\R^{1+N})$,
 we set $f=\lambda u-(\Delta-D_s) u$. Let $G_p(\cdot)$ be the heat semigroup
generated by $\Delta$ on $L^p(\R^N)$ for $1\le p<+\infty$, and on
$C_0(\R^N)$ for $p=+\infty$, respectively. The variation of constants formula
yields
\[
u(t)=\int_{-\infty}^t e^{-\lambda (t-s)}G_p(t-s)f(s)\, ds,
\]
for all $t\in\R$. Using the well known estimate
$ \sqrt{s}\, \|\nabla_x G_p(s)\varphi\|_p \le c\, \|\varphi\|_p $
valid for every $s>0$ and  $\varphi\in L^p(\R^N)$ or $\varphi\in C_0(\R^N)$,
respectively  (where $c=c(N)$ is a constant), we deduce that
\[
\|\nabla_x u(t)\|_p \le \int_{-\infty}^t
  \frac{ce^{-\lambda(t-s)}}{\sqrt {t-s}} \|f(s)\|_p\,ds,
\]
for all $t\in \R$. Young's inequality then implies
\[
\|\nabla_x u\|_p \le \frac{c\sqrt{\pi}}{\sqrt {\lambda}} \,\|f\|_p
\le\frac{c\sqrt{\pi}}{\sqrt {\lambda}}\,(\lambda\,\|u\|_p
+\|(\Delta-D_s) u\|_p),
\]
for each $\lambda>0$. The assertion follows if we take
$\lambda=\|(\Delta-D_s) u\|_p \,\|u\|_p^{-1}$.
\end{proof}

\begin{lem}\label{density}
Assume that $W\in C(\R^{1+N})$ satisfies $W\ge c_0>0$.
Then,  $C_c^\infty (\R^{1+N})$ is dense in $\D_p$ for
$1 \le p \le +\infty$.
\end{lem}
\begin{proof}
Let $\eta$ be a cutoff function  on $\R^{1+N}$ such that $\1_{B(0,1)}\le\eta\le\1_{B(0,2)}$.
Define $\eta_n(t,x)=\eta(t/n,x/n)$ for all $(t,x)\in\R^{1+N}$ and $n\in\N$.
Let $u\in \D_p$. Then, $\eta_n u \to u$ and $W\eta_n u \to Wu$ as $n \to +\infty$ in
$L^p(\R^{1+N})$. Moreover,
\[ (D_s-\Delta)(\eta_n u)= \eta_n (D_s-\Delta) u + u (D_s-\Delta) \eta_n
        - 2 \langle \nabla_x u, \nabla_x \eta_n\rangle.\]
Since the derivatives of $\eta_n$ tend to 0 in the sup-norm as $n\to+\infty$,
the functions $(D_s-\Delta)(\eta_n u)$ converge to $(D_s-\Delta) u $
in  $L^p(\R^{1+N})$. Hence, the
set of all functions in $\D_p$ having compact support is dense in $\D_p$.
On the other hand, if $u \in \D_p$ has compact support, a standard convolution
argument shows the existence of a sequence of smooth functions with compact
support converging to $u$ in $\D_p$, since $W$ is bounded in each
neighborhood  of the support of $u$.
\end{proof}

The next result is again proved for all $p\in [1,+\infty]$.
It will allow us to control the drift term by the heat
operator and the potential.

\begin{prop}\label{interpol1}
Let $W$ be a function satisfying {\rm (A2)}. Then, there exists
a constant $\alpha>0$ (depending only on $N$,
$\beta$, $\gamma,$ $K_\beta$, $K'_\gamma$, $c_0$) such that
\begin{equation}\label{u-nabla-est}
\|W^{\frac{1}{2}}\nabla_x u\|_p \le \eps\, \|(\Delta-D_s) u\|_p
     +\frac{\alpha}{\eps}\,\|Wu\|_p,
\end{equation}
for all  $\eps\in(0,1]$, $1\le p\le +\infty$, and  $u\in \D_p$.
\end{prop}
\begin{proof}
It suffices to show the proposition for test functions $u$.
Lemma~\ref{density} then allows us to extend the result to all $u\in\D_p$
by approximation.
We further can replace $W$ by $W+\lambda$ for some $\lambda\ge0$ such
that (A2) holds for $W+\lambda$ with $K_\beta=K'_\gamma=0$. Since
$W\ge c_0>0$, the estimate
\eqref{u-nabla-est} for $W+\lambda$ implies \eqref{u-nabla-est} for $W$
(with a different $\alpha$). So, we may and will assume that
$K_\beta=K'_\gamma=0$  in (A2).  Hence,
\begin{equation}\label{w-est}
 |D_s W^{-1}|\le \beta \quad\text{and}\quad
     |\nabla_x W^{-\frac{1}{2}}|\le\frac{\gamma}{2}
   \qquad \text{in \ }\R^{1+N}.
\end{equation}
In what follows we write $\nabla$ instead of $\nabla_x$. Our
arguments rely on a localization procedure in space and time.
We set  $\tau:=\tau(s_0,x_0)=(4\beta\ell_1W(s_0,x_0))^{-1}$ for every
given $(s_0,x_0) \in \R^{1+N}$ and a number $\ell_1\ge 1$ to be fixed later.
Since $\tau\le (4\beta W(s_0,x_0))^{-1}$, from \eqref{w-est} it
follows that
\begin{align*}
\frac{4}{5}\, W(s,x_0)^{-1} &\le   W(s_0,x_0)^{-1}
      \le \frac43\, W(s,x_0)^{-1}, \\[1mm]
\frac{\sqrt{3}}{2}\, W(s,x_0)^\frac12 &\le   W(s_0,x_0)^\frac12
             \le \frac{\sqrt{5}}{2}\, W(s,x_0)^\frac12,
\end{align*}
for all $t\in\R$ with $|t-s_0|\le\tau$. We next set $r:=r(s_0,x_0)=
\sqrt{3}(2\ell_2 \gamma W(s_0,x_0)^\frac12)^{-1}$ for a number
$\ell_2\ge 1$ to be chosen later. Note that $r\le
(\ell_2\gamma)^{-1} W(s,x_0)^{-\frac12}$ for every $ t\in
(s_0-\tau,s_0+\tau)$. Estimate \eqref{w-est} then implies
\[
\frac{2\ell_2-1}{2\ell_2}\, W(s,x)^\frac12 \le   W(s,x_0)^\frac12
\le \frac{2\ell_2+1}{2\ell_2}\, W(s,x)^\frac12,
\]
for all $x\in B(x_0,r)$ and $s\in (s_0-\tau,s_0+\tau)$.
We thus obtain
\begin{equation}\label{osc}
\frac{(2\ell_2 -1)\sqrt{3}}{4\ell_2}\, W(s,x)^\frac12 \le
W(s_0,x_0)^\frac12
             \le \frac{(2\ell_2+1)\sqrt{5}}{4\ell_2}\, W(s,x)^\frac12,
\end{equation}
for all $(s,x)$ in the parabolic cylinder
$Q=Q(s_0,x_0):=(s_0-\tau,s_0+\tau)\times B(x_0,r)$. We now
choose functions $\eta \in C_c^\infty (\R^N)$ and $\zeta\in
C_c^\infty(\R)$ such that $\1_{B(x_0,r/2)}\le\eta\le\1_{B(x_0,r)}$,
$\1_{(s_0-\tau/2, s_0+\tau/2)}\le\zeta\le\1_{(s_0-\tau, s_0+\tau)}$,
$|\nabla \eta| \leq c/r$, $|D^2 \eta| \leq c/r^2$ and $|D_s \zeta|
\le c/\tau$ for a constant $c$ independent of  $s_0$, $x_0$, $\tau$
and $r$. We set $Q'=Q'(s_0,x_0):=(s_0-\tau/2,s_0+\tau/2)\times
B(x_0,r/2)$ and denote the $p$-norms on $Q'$ and $Q$ by the
additional indexes $Q'$ and $Q$, respectively, for $1\le p\le
+\infty$. Using \eqref{osc}, Lemma~\ref{interpol},  the definitions
of $r,\tau$ and Young's inequality, we compute
\begin{align}
\|&W^\frac12 \nabla u\|_{p,Q'} \le c\,W(s_0,x_0)^\frac12\,\|\nabla u\|_{p,Q'}
  \le  c\,W(s_0,x_0)^\frac12\,\|\nabla (\zeta \eta u)\|_p \notag\\
  &\le c \, \|(\Delta-D_s)(\zeta \eta u)\|_p^{\frac{1}{2}}
        \,\| W(s_0,x_0) \zeta\eta u\|_p^{\frac{1}{2}}\notag\\
 &\leq  c\,\|Wu\|_{p,Q}^\frac12 \Big(\| (\Delta-D_s)u\|_{p,Q}
        + \frac{1}{r}\, \|\nabla u\|_{p,Q}
   +\Big(\frac{1}{\tau} + \frac{1}{r^2}\Big) \|u\|_{p,Q} \Big)^\frac12\notag\\
&\le  c\,\|Wu\|_{p,Q}^\frac12 \,\| (\Delta-D_s)u\|_{p,Q}^\frac12
    + c\gamma \,\|Wu\|_{p,Q}^\frac12 \, \|W^\frac12 \nabla u\|_{p,Q}^\frac12
      +c(\beta+\gamma^2)\, \|Wu\|_{p,Q}\notag \\
& \le \delta \, \|W^\frac12 \nabla u\|_{p,Q}
       +\varepsilon\, \| (\Delta-D_s)u\|_{p,Q}
       + \frac{c(\delta)}{\varepsilon}\, \|Wu\|_{p,Q}, \label{w-grad-est}
\end{align}
for each $\delta, \varepsilon\in(0,1]$, where  the constants $c$  only depend
on $N$, $b$, $\ell_1$, $\ell_2$, and the last one also on $\delta$,
where $b$ is any number such that  $0<\beta, \gamma\le b$.

In the case $p=+\infty$, we fix $\ell_1=\ell_2=1$ and note that
inequality \eqref{w-grad-est} trivially yields
\[W^\frac12(s_0,x_0)\, |\nabla u(s_0,x_0) |
    \le \delta \, \|W^\frac12 \nabla u\|_\infty
       +\varepsilon\, \| (\Delta-D_s)u\|_ \infty
       + \frac{c(\delta)}{\varepsilon}\, \|Wu\|_\infty.\]
We now fix $\delta=1/2$ and take the supremum
over $(s_0,x_0)\in\R^{1+N}$ of the left hand side. The assertion then follows.

For $p\in[1,+\infty)$, we take advantage of Proposition
\ref{prop-besicovitch}. For this purpose, we fix the parameters
$\ell_1$ and $\ell_2$ in the following way:
\begin{eqnarray*}
\ell_1=\frac{2\ell_2^2 \gamma^2}{3\beta},\qquad\;\,
\ell_2=\max\left\{\sqrt{3}\left
(\frac{1}{2}+\frac{\sqrt{\beta}}{\gamma}\right),1\right\}.
\end{eqnarray*}
Clearly, $\ell_1,\ell_2\ge 1$. Moreover, since
$\sqrt{\frac{\tau(s_0,x_0)}{2}}=\frac{r(s_0,x_0)}{2}$ for any
$(s_0,x_0)\in\R^{1+N}$, the cylinder $Q'(s_0,x_0)$ coincides with
the ball $B_d((s_0,x_0),\varrho(s_0,x_0))$ centered at $(s_0,x_0)$
and with radius $\varrho(s_0,x_0):=\sqrt{3}(4\ell_2\gamma
W(s_0,x_0)^{\frac{1}{2}})^{-1}=\frac12 r(s_0,x_0)$, in the metric $d$  (see
\eqref{metric} and \eqref{cylinder}), whereas $Q(s_0,x_0)$ is
properly contained in $B_d((s_0,x_0),2\varrho(s_0,x_0))$. Further,
using \eqref{w-est} we can easily estimate
\begin{align*}
|W(s,&x)^{-\frac{1}{2}}-W(s_0,x_0)^{-\frac{1}{2}}|\\
&\le
|W(s,x)^{-\frac{1}{2}}-W(s,x_0)^{-\frac{1}{2}}| +
|W(s,x_0)^{-\frac{1}{2}}-W(s_0,x_0)^{-\frac{1}{2}}|\\
& \le |W(s,x)^{-\frac{1}{2}}-W(s,x_0)^{-\frac{1}{2}}| +
|W(s,x_0)^{-1}-W(s_0,x_0)^{-1}|^{\frac{1}{2}}\\
& \le \frac{\gamma}{2}|x-x_0|
+\sqrt{\beta}|s-s_0|^{1/2}
\le \left (\frac{\gamma}{2}+\sqrt{\beta}\right
)d((s,x),(s_0,x_0)),
\end{align*}
for any $(s,x), (s_0,x_0)\in\R^{1+N}$. The choice of $\ell_2$
implies that the function $\varrho$ is Lipschitz continuous in
$\R^{1+N}$ with Lipschitz constant not greater than $1/4$. Hence,
Proposition \ref{prop-besicovitch} guarantees the existence of a
countable covering $Q_k'=Q'(s_k,x_k)$ of $\R^{1+N}$ such each
$(s,x)\in\R^{1+N}$ is contained in at most
$K(N)$ of the cylinders $Q_k=Q(s_k,x_k)$, for some integer
$K(N)$. Inequality \eqref{w-grad-est} now implies that
\begin{align*}
\|W^\frac12 \nabla u\|_p^p &\le \sum_{k=1}^{+\infty}
           \|W^\frac12 \nabla u\|_{p,Q'_k}^p \\
 &\le 3^{p-1}\sum_{k=1}^{+\infty} \Big(
      \delta^p \,\|W^\frac12 \nabla u\|_{p,Q_k}^p
        +\varepsilon^p\, \| (\Delta-D_s)u\|_{p,Q_k}^p
       + \frac{c(\delta)^p}{\varepsilon^p}\, \|Wu\|_{p,Q_k}^p\Big)\\
  &\le 3^{p-1} K(N) \Big(\delta^p \,\|W^\frac12 \nabla u\|_{p}^p
        +\varepsilon^p\, \| (\Delta-D_s)u\|_{p}^p
       + \frac{c(\delta)^p}{\varepsilon^p}\, \|Wu\|_{p}^p\Big).
\end{align*}
Fixing $\delta=(3K(N))^{-1}$, we get the assertion also for
 $p\in[1,+\infty)$.
\end{proof}

\begin{rem}\label{alpha-rem}
The above proof shows the following fact (cf.\ the remarks after
\eqref{w-grad-est}). Assume that (A2) holds for some
$\beta,\gamma\in(0,b]$ with $K_\beta=K'_\gamma=0$. Let $p=+\infty$.
Then the constant
$\alpha$ in Proposition~\ref{interpol1} only depends on $N$ and $b$.
\end{rem}

Assume that (A1) holds and fix $p\in (1,+\infty)$. It is known that
the realization in $L^p(\R^{1+N})$ of
the operator $\Div_x(a\nabla_x) -D_s$ with domain $ W^{1,2}_p(\R^{1+N})$
has a nonempty resolvent set, cf.\ \cite[Corollary~2.6]{DLS}.
This fact easily implies  that
 there exists a constant $C_p^0>0$ with
\begin{equation}\label{parab-reg}
\begin{split}
\frac{1}{C^0_p}\,(\|(\Delta-D_s)u\|_p+\|u\|_p)
  &\le\|(\Div_x(a\nabla_x) -D_s)u\|_p +\|u\|_p \\
  &\le C_p^0\, (\|(\Delta-D_s)u\|_p+\|u\|_p)
\end{split}\end{equation}
for all $u\in W^{1,2}_p(\R^{1+N})$.

\begin{cor}
\label{coroll-3.6}
Let $1<p<+\infty$ and assume  that {\rm (A1)}--{\rm (A4)} hold.
We then have
\[ \|W^\frac12 \nabla_x u\|_p\le \eps\, \|\cL u\|_p + \frac{c}{\eps}\,\|Wu\|_p,
\]
for every  $\eps\in (0,\eps_0]$, $u\in \D_p$ and some
constants $c, \eps_0>0$ only depending on  $C_p^0$ $($see \eqref{parab-reg}$)$
and the  constants in {\rm (A1)}--{\rm (A4)}.
\end{cor}
\begin{proof}

For all $u\in \D_p$ and $\eps\in (0,1]$, Proposition~\ref{interpol1} and
\eqref{parab-reg} imply that
\begin{align*}
\|W^\frac12 \nabla_x u\|_p
   &\le \eps\,\|(\Delta-D_s)u\|_p+\frac{\alpha}{\eps}\,\|Wu\|_p\\
   &\le c\eps\, \|\cL u - F\cdot \nabla_x u +Vu\|_p
               +\frac{c}{\eps}\,\|Wu\|_p\\
    &\le  c\eps\, \|\cL u\|_p +  c\eps\, \|W^\frac12 \nabla_x u\|_p
             +\frac{c}{\eps}\,\|Wu\|_p,
\end{align*}
where the constants $c$ only depend on the constants in
\eqref{parab-reg} and in (A1)--(A4). The assertion  follows
if we take a sufficiently  small $\eps>0$.
\end{proof}

We now come to the crucial \emph{a priori} estimate.

\begin{prop} \label{prop-dom-p}
Let $p\in (1,+\infty)$. Assume that the assumptions {\rm (A1)}--{\rm (A5)} and
the inequality
\begin{equation}\label{small}
\frac{\theta}{p} + (p-1)\Big(\frac{\beta+\gamma\kappa}{p}
    +\frac{\gamma^2M^2}{4}\Big)<1
\end{equation}
hold, where $M=\sup\{\|a(s,x)^\frac12\|: (s,x)\in\R^{1+N}\}$. Then, there exists
a  constant $C_p>0$
(only depending on $C_p^0$, $M$ and the constants in {\rm (A1)}--{\rm (A5)}) such that
\begin{equation}\label{apriori-est}
C_p^{-1}\,\|u\|_{\D_p} \le\|\cL u\|_p +\|u\|_p  \le C_p\, \|u\|_{\D_p},
\qquad\;\, u\in \D_p.
\end{equation}
\end{prop}
\begin{proof}
We observe that the second estimate in \eqref{apriori-est}
follows from Proposition~\ref{interpol1}. Concerning the first estimate,
we can restrict ourselves to the case where $K_\beta=K'_\gamma=0$ in (A2).
Indeed, in the general  case it suffices to fix a large $\lambda>0 $ such that $W+\lambda$
satisfies (A2) with $K_\beta=K'_\gamma=0$. The established estimate
for  the operator $\cL-\lambda$ with the potential $V+\lambda$ then yields
\[ \|u\|_{\D_p} \leq c\, (\|\cL u-\lambda u\|_p+\|u\|_p)
  \le c\,(\|\cL u\| +\| u\|_p),
  \]
for some constants only depending on $C_p^0$, $M$ and the constants in
(A1)--(A5). Moreover, in view of Lemma~\ref{density}, it is enough to prove
the first inequality in \eqref{apriori-est} for test functions $u$ .

So, let us fix $u\in C^{\infty}_c(\R^{1+N})$. At first we take $p\in[2,+\infty)$.
We set $f:=-{\mathscr L}u$,  multiply this equality by the function
$W^{p-1}|u|^{p-2}u$ and integrate by parts over $\R^{1+N}$. We
then obtain (writing $\Div$ and $\nabla$ for, respectively, $\Div_x$ and $\nabla_x$)
\begin{align}
&\int_{\R^{1+N}}f u|u|^{p-2}W^{p-1}\,ds\,dx \label{longformula}\\
&= \frac{1}{p}\int_{\R^{1+N}} (D_s|u|^p)\,W^{p-1}\,ds\,dx   + \int_{\R^{1+N}}
 \langle a \nabla u, \nabla (u|u|^{p-2} W^{p-1})\rangle\,ds\,dx \notag \\
&\quad -\frac{1}{p}\int_{\R^{1+N}} F\cdot(\nabla |u|^p)\,W^{p-1}\,ds\,dx
        +\int_{\R^{1+N}} VW^{p-1} |u|^p\,ds\,dx\notag\\
&= \frac{1-p}{p} \int_{\R^{1+N}} |u|^p W^{p-2} D_sW\,ds\,dx
   + (p-1)\int_{\R^{1+N}}   \langle a \nabla u, \nabla u\rangle\,
         |u|^{p-2} W^{p-1} \,ds\,dx\notag\\
&\quad + (p-1)\int_{\R^{1+N}} \langle a \nabla u, \nabla W\rangle\,
         u |u|^{p-2} W^{p-2} \,ds\,dx\notag\\
 &\quad   + \frac{p-1}{p} \int_{\R^{1+N}} (F\cdot \nabla W) \, |u|^p W^{p-2}\,ds\,dx
 + \int_{\R^{1+N}} \Big(V+\frac{1}{p}\Div F\Big) W^{p-1} |u|^p\,ds\,dx.\notag
\end{align}
These equations are also valid if $p\in (1,2)$, but then
the  integration by parts needs  some justification given by \cite{MS}. From now
on we thus take  $p\in(1,+\infty)$.
Assumptions (A3) and (A5) further  yield
\begin{align} \label{aux-pot}
V+\frac{1}{p}\Div\,F\ge \Big(1-\frac{\theta}{p}\Big)W.
\end{align}
Formulas \eqref{longformula} and \eqref{aux-pot},
H\"older's inequality, and conditions (A2) and (A4) imply
\begin{align*}
\Big(1-\frac{\theta}{p}\Big) & \|W u\|^p_p
  + (p-1)\int_{\R^{1+N}} |a^\frac12 \nabla u|^2 |u|^{p-2} W^{p-1}\,ds\,dx\\
&\le \|f\|_p \,\|Wu\|_p^{p-1} + \frac{\beta (p-1)}{p} \,\|Wu\|^p_p
    +  \frac{\gamma \kappa(p-1)}{p}\,\|Wu\|^p_p\\
&\quad\; +  (p-1)\int_{\R^{1+N}} | a^\frac12 \nabla u|\,
    |a^\frac12 \nabla W| \,|u|^{p-1} W^{p-2}\,ds\,dx.
\end{align*}
Using again (A2) and H\"older's inequality,  the last summand can be
estimated by
\begin{align*}
 (&p-1)\gamma M\int_{\R^{1+N}}|a^\frac12 \nabla u|\,W^{p-\frac12}
            |u|^{p-1}\,ds\,dx\\
&\le (p-1)\gamma M \Big(\int_{\R^{1+N}} W^p |u|^p \,ds\,dx\Big)^\frac12
    \Big(\int_{\R^{1+N}} |a^\frac12 \nabla u|^2
     |u|^{p-2} W^{p-1}\,ds\,dx\Big)^\frac12\\
 &=: (p-1)\gamma M\,AB.
\end{align*}
By means of Young's inequality, we then deduce
\begin{equation}\label{ab-est}
\Big[1-\frac{\theta}{p} -\frac{p-1}{p} (\beta+\gamma\kappa)-\eps\Big] A^2
   + (p-1) B^2 - (p-1)\gamma M AB \le c(\eps) \|f\|_p^p,
\end{equation}
for each $\eps>0$ and some $c(\eps)>0$. Due to assumption
\eqref{small}, we can fix $\eps>0$ such that the left hand side of \eqref{ab-est}
is larger than  $\eta(A^2+B^2)$ for some $\eta>0$. So, we have shown that
\begin{equation}\label{full-apriori}
 \|Wu\|_p^p + \int_{\R^{1+N}} |a^\frac12 \nabla u|^2|u|^{p-2} W^{p-1}\,ds\,dx
  \le c\, \|f\|_p^p =c\,\|\cL u\|^p_p.
\end{equation}
Here and below
the constants $c$ only depend on $M$, $C_p^0$ (see \eqref{parab-reg}) and the constants in
(A1)--(A5).  Assumption (A4), Corollary~\ref{coroll-3.6} and
the  estimate \eqref{full-apriori} further yield
\begin{align*}
\|F\cdot \nabla u\|_p\le \kappa\, \|W^\frac12 \nabla u \|_p
    \le c\,(\|\cL u\|_p +  \|Wu\|_p) \le c\,\|\cL u\|_p\,.
\end{align*}
Using \eqref{parab-reg}, the last inequality, \eqref{full-apriori} and recalling that
$V\le c_1 W$, we get
\begin{align*}
\|u\|_{\D_p} \le c\,(\|\cL u -F\cdot \nabla u+ Vu\|_p + \|Wu\|_p)
                   \le  c\, \|\cL u\|_p\,,
 \end{align*}
 which is the remaining part of \eqref{apriori-est}.
\end{proof}

We now want to treat the inhomogeneous parabolic equation
\begin{align}\label{parab-eq}
D_su(s)= \Div_x(a(s)\nabla_x u(s))+F(s)\cdot \nabla_x u(s)
      -V(s)u(s) +f(s), \quad s\in \R,
\end{align}
on $\R^{N}$. For this purpose,
we define the operator $L_p u=\cL u$ with $D(L_p)=\D_p$ in
$L^p(\R^{1+N})$, where $1<p<+\infty$. In the next theorem
we identify $L^p(\R^{1+N})$ with $L^p(\R, L^p(\R^N))$ and we use
the following concepts. An \emph{evolution family} $G(s,r)$, $s\ge r$,
is a family
of bounded operators on a Banach space $X$ such that
 \[ G(t,s)G(s,r)=G(t,r), \quad G(s,s)=I, \quad (s,r)\mapsto G(s,r)
 \text{ is strongly continuous}, \]
 for $r,s,t\in\R$ with $t\ge s\ge r$.  The corresponding
\emph{evolution semigroup} on $L^p(\R,X)$ is given by
\[ (S(t)f)(s)=G(s,s-t) f(s-t),\]
 for $f\in L^p(\R^{1+N})$, $s\in\R$ and  $t\ge0$.
(See e.g.\ \cite{ChiLa} or \cite{Sch-survey}.)

\begin{thm}\label{Lp-thm}
Let $p\in(1,+\infty)$ and assume that conditions {\rm (A1)}--{\rm (A5)} and
\eqref{small} are satisfied. Then, the following assertions hold.
\begin{enumerate}[(a)]
\item
The operator $L_p$ generates a positive and
contractive evolution semigroup $S_p(\cdot)$ on $L^p(\R^{1+N})$
induced by an evolution family $G_p(s,r)$,
 $s\ge r$, of positive contractions on $L^p(\R^N)$.
\item
We set $u:= G_p(\cdot,r)\varphi$ for every $\varphi\in  L^p(\R^{N})$
and $r\in \R$. Then, $u\in W^{1,2}_p((a,b)\times \R^N)$,
$Vu\in L^p((a,b)\times\R^N)$ and $D_s u(s) = \A(s)u(s)$ for $s\in (a,b)$
and each interval $[a,b]\subset  (r,+\infty)$.
Moreover, for each $f\in L^p(\R^{1+N})$ there exists a unique $u\in \D_p$
satisfying \eqref{parab-eq}, namely
\[
u(s)=-L_p^{-1}f(s)=\int_{-\infty}^s G_p(s,r)f(r)\,dr,\qquad s\in\R.\]
\item
Let conditions {\rm (A5)} and \eqref{small} also hold
 for some $q\in (1,+\infty)$.  Then, $S_p(\cdot)$ and $S_q(\cdot)$ (resp. $G_p(\cdot,\cdot)$ and $G_q(\cdot,\cdot)$)
 coincide in $L^p(\R^{1+N})\cap L^q(\R^{1+N})$ (resp. in $L^p(\R^N)\cap L^q(\R^N)$).
 \end{enumerate}
\end{thm}

\begin{proof}
Being rather long, we split the proof in three steps.
\par
{\em Step 1}.
Due to Proposition~\ref{prop-dom-p}, the operator $L_p$ is closed on $\D_p$.
Hence, Lemma~\ref{density} implies that the space $C_c^\infty(\R^{1+N})$ is
a core for $L_p$.  The dissipativity of $L_p$ now follows from
Lemma~\ref{dissip}.  As a result,  $L_p$ generates a contraction semigroup
on $L^p(\R^{1+N})$ if $I-L_p$ is invertible, thanks to the Lumer--Phillips
 theorem (see e.g., \cite[Theorem II.3.15]{EN}).  We employ the operators
$\cL_\tau =\Div_x(a\nabla_x) +\tau F\cdot\nabla_x -V-D_s$
for  $\tau\in[0,1]$. Since these operators satisfy (A1)--(A5) and \eqref{small}
with the
same constants, Proposition~\ref{prop-dom-p} combined with the dissipativity
of $\cL_\tau$ yield that
\[ \|u\|_{\D_p}\le c\, (\|\cL_\tau u\|_p + \|u\|_p)
\le c\, (\|\cL_\tau u-u \|_p + \|u\|_p) \le c\, \|u-\cL_\tau u\|_p, \]
for every $u\in\D_p$,
with constants independent of $\tau\in[0,1]$. Hence, $I-L_p$ is invertible
if $I-\cL_0:\D_p\to L^p(\R^{1+N})$ is invertible, see e.g., \cite[Theorem~5.2]{GT}.
Observe that $\cL_0$ has no drift term.  We use the
Yosida approximations $ V_\eps= V (1+\eps V)^{-1}$ and
$W_\eps= W (1+\eps W)^{-1}$ of $V$ and $W$, respectively, where $\eps\in(0,1]$.
It is easy to check that the potential $W_\eps$ and the coefficients
of $\cL_{0,\eps} = \Div_x(a\nabla_x) +V_\eps -D_s$ also  satisfy (A1)--(A5)
and \eqref{small} with the same constants (except that
one has to replace $c_0$ by $c_0(1+c_0)^{-1}$). Moreover, $\cL_{0,\eps}$ with
domain $W^{1,2}_p(\R^{1+N})$ generates a contraction semigroup
on $L^p(\R^{1+N})$.
For every $f\in L^p(\R^{1+N})$ and $\eps\in(0,1]$,  we can thus define
$u_\eps=(I-\cL_{0,\eps})^{-1}f\in W^{1,2}_p(\R^{1+N})$, i.e.,
$u_\eps-\cL_{0,\eps} u_\eps=f$. From  the dissipativity of $\cL_{0,\eps}$ and
 Proposition~\ref{prop-dom-p} we deduce that $\|u_\eps\|_p\le \|f\|_p$ and
\[  \|W_\eps u_\eps\|_p+ \|u_\eps\|_{W^{1,2}_p(\R^{1+N})}
    \le c \,(\|\cL_{0,\eps} u_\eps\|_p +\|u_\eps\|_p)\le c\, \|f\|_p\,, \]
where the constants $c$ do not depend on $\eps\in(0,1]$. So, we find
a sequence $(u_{\eps_n})$ converging weakly to a function
$u\in  W^{1,2}_p(\R^{1+N})$. A subsequence converges  in $L^p_{\loc}(\R^{1+N})$,
so that we may assume that $u_{\eps_n}\to u$ a.e.\ in $\R^{1+N}$.
This fact implies that
$\|W u\|_p\le c\, \|f\|_p$, and hence $u\in \D_p$. Finally,
we can pass to the limit
(in the sense of distributions) in the equation $u_\eps-\cL_{0,\eps} u_\eps=f$,
obtaining $u-\cL_0 u=f$. Consequently, $I-\cL_0$ with domain $\D_p$ is
invertible so that $L_p$ generates a contraction semigroup $S(\cdot)$ on
$L^p(\R^{1+N})$.

Let us now check that $T_p(\cdot)$ is an evolution semigroup
and that the associated evolution operator $G_p(\cdot,\cdot)$ is contractive.
For this purpose. we begin by noting that $\D_p$ is a dense subset of $C_0(\R, L^p(\R^N))$
and $(I-L_p)^{-1}$ is continuous from $L^p(\R,L^p(\R^N))$ into $C_0(\R,L^p(\R^N))$.
Moreover,
\begin{equation*}
L_p(\varphi f)=\varphi L_pf-\varphi'f,
\end{equation*}
for all $f\in {\mathscr D}_p$ and $\varphi\in C^1_c(\R)$.
Theorem~3.4 of \cite{RRS} now yields the existence of an evolution family
$G_p(s,r)$, $s\ge r$, such that $(S_p(t)f)(s)=G_p(s,s-t) f(s-t)$ for
$f\in L^p(\R^{1+N})$, $s\in\R$, and $t\ge0$ (see also
\cite[Theorem~4.2]{Sch-survey} and the  references therein).
By \cite[Formula (3.3)]{RRS}, for all $s>r$ it holds that
$\|G_p(s,r)\|_{L(L^p(\R^N))}\le\|T_p(s-r)S_0(r-s)\|_{L(L^p(\R^{1+N}))}$,
where $S_0(\cdot)$ is the semigroup of left translations (i.e., $S_0(t)f=f(\cdot-t)$ for $t\in\R$ and
$f\in L^p(\R^{1+N})$). Since both $T_p(\cdot)$ and $S_0(\cdot)$ are contractive semigroups, the contractivity
of the operator $G_p(s,r)$  on $L^p(\R^N)$ for all $s\ge r$ follows at once.
\par
{\em Step 2}. By Step 1, the operator $\delta I-L_p$ is invertible for all
$\delta>0$. On the other hand,
for sufficiently small $\delta\in (0,c_0)$ also the operator $\cL+\delta I$
satisfies assumptions (A1)--(A5) and \eqref {small}
for the  potentials $V-\delta$ and  $W-\delta$, different constants
$c_0, c_1, K_\beta, K_\gamma'$ and slightly increased $\alpha,
\beta,\theta$ and $\kappa$. As a consequence,
also the operator $L_p$ is invertible, whence the second part of
assertion (b) follows. (Use \cite[p. 68]{ChiLa} for the formula
for $L_p^{-1}$.)

Let now fix $\varphi\in  L^p(\R^N)$, $r\in\R$, and
$[a,b]\subset (r,+\infty)$. Take a function $\phi \in C_c^\infty(\R)$
with $\phi\equiv 1$ on $[a,b]$ and support contained in  $(r,+\infty)$. Define
the function $u\in L^p(\R, L^p(\R^N))$
by $u(s)=\phi(s) G(s,r)\varphi $ for $s\ge r$ and $u(s)=0$ for $s<r$.
As in  \cite[p. 64]{ChiLa} one sees that  $u\in D(L_p)=\D_p$ and
 $L_p u(s)= -\phi'(s) G(s,r)\varphi$ for $s\ge r$.
So, we have established assertion (b).
\par
{\em Step 3.}
It remains to show part (c) and the asserted positivity in (a).
Theorem~3.4 of \cite{MPRS} states that
the operator  $A_p(s)=(\A(s), W^{2}_p(\R^N)\cap D(W(s)))$  generates
a contraction semigroup $(e^{tA_p(s)})_{t\ge0}$ on $L^p(\R^N)$
for each $s\in\R$. Moreover, $A_p(s)$  admits
$C_c^\infty(\R^N)$ as a core. This semigroup is positive because of Theorems~3.3
and 4.1 of \cite{AMP}, see also  \cite[Proposition~6.1]{MPRS}.
As in \cite[Paragraph~III.4.13]{EN}, one verifies that the multiplication operator
$A_p(\cdot)$ with maximal domain
\[D(A_p(\cdot))=\{u\in L^p(\R^{1+N}):u(s)\in D(A_p(s)) \text{ for a.e. }
     s\in\R,\ A_p(\cdot)u\in  L^p(\R^{1+N})\}\]
generates the  semigroup $M(\cdot)$
on $L^p(\R^{1+N})$ given by $M(t)f=e^{tA_p(\cdot)}f(\cdot)$,
which is positive and contractive. Moreover, the first derivative $-D_s$
with domain $W^{1}_p(\R, L^p(\R^{N}))$ generates the  positive contraction
semigroup $S_0(\cdot)$ on $L^p(\R^{1+N})$. Observe that
$D(A_p(\cdot))\cap D(-D_s)=\D_p$ and $L_p=A_p(\cdot)-D_s$.
Therefore, the Lie-Trotter product formula (see \cite[Corollary~III.5.8]{EN})
implies the positivity of $S(t)$, and thus of $G(s,r)$, for all $t\ge0$ and
$s\ge r$. The semigroups $M(\cdot)$ and  $S_0(\cdot)$  on $L^p(\R^{1+N})$
for different values of $p$ coincide on the intersections
of the $L^p$ spaces (see \cite[Theorem~3.3]{AMP} or \cite[Lemma~4.3]{MPRS}).
Hence, the  Lie--Trotter product formula further shows that
the respective evolution semigroups,
and thus the evolution families, coincide.
\end{proof}

In the following remark we indicate that
Theorem~\ref{Lp-thm} cannot be deduced from known
results in the autonomous case.

\begin{rem}\label{AT-rem}
Under the assumptions of Theorem~\ref{Lp-thm} we
define the operator $A(s)$ in $L^p(\R^N)$ by setting
$A(s)\varphi=\A(s)\varphi$
for $\varphi\in D(A(s)):=\{v\in W^2_p(\R^N): W(s)v\in L^p(\R^N)\}$,
$s\in\R$ and $p\in(1,+\infty)$.
Theorem~3.4 and Proposition~6.5 of \cite{MPRS} then state that the operators
$A(s)$  are sectorial and have  maximal $L^p$-regularity (with uniform
constants). We refer the reader to  \cite{KW}
for the concept of maximal $L^p$-regularity. In addition,
assume for a moment that the operators $A(s)$ satisfy
the Acquistapace--Terreni condition; i.e., that there are constants $L\ge0$ and $\mu,\nu\in (0,1]$ such that $\mu+\nu>1$ and
\begin{align}\label{at}
\|\lambda^\nu A(t)(\lambda-A(t))^{-1}(A(t)^{-1}-A(s)^{-1})\|
  \le L\,|t-s|^\mu
\end{align}
holds for all $\lambda>0$ and $t,s\in\R$, see \cite{acqui,AT}.
Corollary~2.6 in \cite{DLS} then implies that for some
$\omega\ge0$ the operator $A(\cdot)-D_s-\omega I$ with domain $\D_p$
is invertible in $L^p(\R^{1+N})$. We point out that this fact is the
crucial point  of the proof of Theorem~\ref{Lp-thm}. However, the
Acquistapace--Terreni condition does not follow from the assumptions
of Theorem~\ref{Lp-thm}, as we now show by a simple example.

Let $a=I$, $F=0$, $N=1$, $p=2$, and set $W(s,x)=V(s,x)=\exp(\exp(s+x))$
for $(s,x)\in\R^2$. It is easy to check that the assumptions
(A1)--(A5) and \eqref{small} hold in this case.
On the other hand, if \eqref{at} were true,
then $D(A(0))=W^2_2(\R^2)\cap D(V(0))$ would be contained in the real
interpolation space $(X, D(A(s)))_{\nu,\infty}$ which is embedded
into $D(V(s)^\alpha)$ for all $s\in\R$ and $\alpha\in (0,\nu)$.
(See e.g.\ \cite{lunardi} for basic facts on interpolation theory.).
Given such an $\alpha\in(0,\nu)$ take $s>0$ such that $\alpha e^s=2$.
Choose a function $\chi\in C^2(\R)$ which vanishes on $\R_-$ and is equal
to 1 on $[1,+\infty)$. Set $v(x)=\chi(x) \exp(-\frac32 e^x)$ for $x\in\R$.
It is straightforward to verify that $v\in D(A(0))$ but
$v\notin D(V(s)^\alpha)$, so that \eqref{at} has to be violated
in this example.
\end{rem}

In order to apply  Theorem~\ref{Lp-thm} to the parabolic Ornstein--Uhlenbeck operator
$\G_O$, we have to study the mapping properties of the isomorphism
$M_p:L^p(\R^{1+N})\to L^p(\R^{1+N},\nu)$,  see  \eqref{Mp}, on the space
\[\D_{p,O}=\{u\in W^{1,2}_p(\R^{1+N}): |x|^2 u\in L^p(\R^{1+N})\},\]
endowed with the norm $\|u\|_{\D_{p,O}}=\|u\|_{W^{1,2}_p(\R^{1+N})}+\|\,|x|^2u\|_p$,
i.e., the space $\D_p$ for  the potential $W_O(x)=c_0+k_1\,|x|^2$
from \eqref{VF-cond}.

\begin{lem}\label{Mp-lem}
Assume that Hypothesis $\ref{hyp1}$ holds and let $p\in (1,+\infty)$.
 Then, the map $M_p$ defined in \eqref{Mp} induces an
 isomorphism from $\D_{p,O}$ onto $W^{1,2}_p(\R^{1+N},\nu)$.
\end{lem}
\begin{proof}
We have to prove that the restrictions $M_p:\D_{p,O}\to
W^{1,2}_p(\R^{1+N},\nu)$ and $M_p^{-1}: W^{1,2}_p(\R^{1+N},\nu)\to\D_{p,O}$
are well-defined and  continuous. Concerning $M_p$, it suffices to show
\begin{align}\label{mp-est}
\|M_p u\|_{W^{1,2}_p(\R^{1+N},\nu)}
  \le c\, (\|u\|_{W^{1,2}_p(\R^{1+N})} + \|\,|x|^2 \,u\|_p),
\end{align}
for a constant $c$ and all $u\in C_c^\infty(¸\R^{1+N})$,
because of Lemma~\ref{density}. We further recall that the norm
of the functions $|x|\, |\nabla_x u|$ in $L^p(\R^{1+N})$ can be
controlled by the norm of $u$ in $\D_{p,O}$, due to Corollary \ref{coroll-3.6}.
The formulas \eqref{change-unknown},  \eqref{D-Phi} and \eqref{Ds-Phi}
combined with Lemma~\ref{prop-1} now easily imply \eqref{mp-est}.

To establish the continuity of the operator $M_p^{-1}:W^{1,2}_p(\R^{1+N},\nu)\to\D_{p,O}$
we first note that the space $C^{\infty}_c(\R^{1+N})$ is dense in $W^{1,2}_p(\R^{1+N},\nu)$. This fact can
be shown as in Lemma \ref{density}. It remains to prove that
\begin{align}\label{mp-inv-est}
\|M_p^{-1} v\|_{W^{1,2}_p(\R^{1+N})}  + \|\,|x|^2 \,M_p^{-1} v\|_p
  \le c\, \|v\|_{W^{1,2}_p(\R^{1+N},\nu)},
\end{align}
for a constant $c$ and all $v\in C_c^\infty(\R^{1+N})$.
For the derivatives of $u:=M^{-1}_p v$ one can
obtain expressions similar to those in \eqref{change-unknown}.
Hence,  we have to
dominate the norms in $L^p(\R^{1+N})$ of the functions $|x| u$, $|x|^2 u$ and
$|x|\, |M_p^{-1}(D_i v)|$ by $\|v\|_{W^{1,2}_p(\R^{1+N},\nu)}$.
We  prove below that there exists $c>0$ such that
\begin{equation}\label{incl-2}
\begin{split}
\intrn |x|^p |M_p^{-1} v|^p\,ds\,dx &\le c \intrn |x|^p |v|^p\, d\nu\\
  &\le c\sum_{j=1}^N\intrn |D_j v|^p\,d\nu + c\intrn |v|^p\,d\nu.
\end{split}\end{equation}
After \eqref{incl-2} has been shown, we can apply
this inequality   also to the functions $D_i v$ and $x_i v$,
where $i=1,\ldots,N$. In this way we derive \eqref{mp-inv-est}.

To show  \eqref{incl-2},  let $v$ be a test function.
At first,  Lemma~\ref{prop-1} yields
\begin{align*}
\intrn |x|^p |M_p^{-1} v|^p\,ds\,dx
&= (2\pi)^\frac{N}{2} \intrn |x|^p |v|^p (\det Q_s)^\frac12 \,d\nu  \\
&\le (2\pi C_2)^\frac{N}{2} \intrn |x|^p |v|^p \,d\nu .
\end{align*}
To check the second part of \eqref{incl-2}, we first deduce from
Lemma~\ref{prop-1} the estimates
\begin{align*}
\intrn |x|^p|v|^p\,d\nu
&\le  (2\pi C_1)^{- \frac{N}{2}} \intrn |x|^p|v|^p\,
         e^{-\frac{1}{2}\langle Q_s^{-1}x,x\rangle}\,ds\, dx \notag \\
&\le \frac{C_2^p}{(2\pi C_1)^{\frac{N}{2}}} \intrn |v|^p
 |Q_s^{-1}x|^p \,e^{-\frac{1}{2}\langle Q_s^{-1}x,x\rangle}\,ds\,dx.
\end{align*}
On the other hand, \cite[Lemma~7.1]{MPRS} implies that
\[ \int_{\R^N}\! |v(s,x)|^p |Q_s^{-1}x|^p
   e^{-\frac{1}{2}\langle Q_s^{-1}x,x\rangle}dx
  \le c \int_{\R^N}\! (|v(s,x)|^p+ |\nabla_x v(s,x)|^p)
    e^{-\frac{1}{2}\langle Q_s^{-1}x,x\rangle}dx,
    \]
for a constant $c>0$ and every $s\in\R$. We integrate this inequality
with respect to $s\in\R$ and use once more Lemma~\ref{prop-1}.
As a result, \eqref{incl-2} holds.
\end{proof}

We come now to our second main result which describes the domain
of the parabolic Ornstein-Uhlenbeck operator $\G=\A_O(\cdot)-D_s$
in the Lebesgue space
with the family of invariant measures. This fact has immediate
consequences on the regularity properties of the equation \eqref{ou-line}.

\begin{thm}\label{ou-thm}
Let $p\in(1,+\infty)$ and assume that Hypothesis $\ref{hyp1}$ holds.
 Then, the operator $G_p=(\G,W^{1,2}_p(\R^{1+N},\nu))$ generates a positive
contraction semigroup $T(\cdot)$ on $L^p(\R^{1+N})$. This semigroup
is given by $(T(t)f)(s)=G_O(s,s-t) f(s-t)$ for $f\in L^p(\R^{1+N},\nu)$,
 $s\in\R$, $t\ge0$, and the positive and contractive
 evolution family $G_O(s,r)$, $s\ge r,$ solving
 \eqref{pb-3}.  Further,
$u:= G_O(\cdot,r)\varphi\in W^{1,2}_p((a,b)\times \R^N, \nu)$
and $D_t u(s) = \A_O(s)u(s)$ for every $\varphi\in  L^p(\R^{N})$, $r\in \R$,
$[a,b]\subset  (r,+\infty)$ and $s\in (a,b)$.
 Finally, for each
$f\in L^p(\R^{1+N},\nu)$ there exists a unique $u\in D(G_p)$ satisfying
\eqref{ou-line}, namely $u:=(I-G_p)^{-1}f$.
\end{thm}
\begin{proof}
We can apply Theorem~\ref{Lp-thm} to the operator $\cL_O-\lambda I$,
see \eqref{LO} and \eqref{VF-cond}.
Theorem~\ref{Lp-thm} and Lemma~\ref{Mp-lem} thus imply that the
operator $G_p-\lambda I$ with domain
$D(G_p)=M_p(\D_p)=W^{1,2}_p(\R^{1+N},\nu)$  generates
a positive semigroup on $L^p(\R^{1+N},\nu)$. Moreover, $G_p$ extends
the operator $\G$ defined on test functions. As mentioned in the proof of
Lemma~\ref{Mp-lem}, test functions are dense in $W^{1,2}_p(\R^{1+N},\nu)$
and thus they are a core for $G_p$. In view of \cite{LZ}, $G_p$ then generates the
evolution semigroup $T(\cdot)$ corresponding to $G_O$, as described in the
introduction. (Note that Hypothesis 1.1(iv) in \cite{LZ} is needed only to
guarantee the continuity of the function $G(s,r)f$ with respect to $r$, when $f\in C_b(\R^N)$ and
$G(s,r)$ is the evolution operator
associated with the class of nonautonomous Kolmogorov operators therein considered;
in our situation that assumption is not needed since the continuity of the function $G_O(s,r)f$ with respect to
$r$ is clear since we have an explicit formula for this function, see \cite{dpl}.)
 This semigroup is contractive. The remaining assertions can be
shown as in Theorem~\ref{Lp-thm}.
\end{proof}

\section{Operators with dominating potential for $p=1,+\infty$.}
\label{sect-4}

In this section we extend Theorem~\ref{Lp-thm} to the borderline cases
$p=1,+\infty$. We set $L_1=(\cL,\D_1)$ on $L^1(\R^{1+N})=L^1(\R,L^1(\R^N))$
and $L_\infty=(\cL,\D_\infty)$ on $C_0(\R^{1+N})=C_0(\R, C_0(\R^N))$.
 Note that in these cases we cannot expect to replace
$\D_1$ and $\D_\infty$ with the intersection $W^{1,2}_p(\R^{1+N})\cap D(W)$ and
$W^{1,2}_{\infty}(\R^{1+N})\cap D(W)$, respectively.
To avoid some technical problems,  we restrict ourselves to the case of the
Laplacian, where $a(s)=I$ for all $s\in\R$.

We need in the next proof some properties
of the operator $\Delta-D_s$ on $L^1(\R,L^1(\R^N))$. Consider the semigroup
$G(\cdot)$ generated by the Laplacian on $L^1(\R^N)$ and let
$(V(t)f)(s) =G(t)f(s-t)$ on $L^1(\R,L^1(\R^N))$
be the induced evolution semigroup, which is positive
and contractive. The generator $H$ of the semigroup $V(\cdot)$
is the closure of $\Delta-D_s$ defined on the intersection of the domains
of $\Delta$ and of $D_s$ in $L^1(\R,L^1(\R^N))$, see e.g.,\
\cite[Remark~2.35]{ChiLa}. The semigroup $V(\cdot)$ leaves invariant
the Schwartz space of rapidly decreasing functions
$f:\R^{1+N}\to\R$ which, thus, is a core of $H$.
In view of Lemma~\ref{density}, it follows that
\[D(H)= \{u\in L^1(\R,L^1(\R^N)): (\Delta-D_s)u\in L^1(\R,L^1(\R^N))\}
=:D(\Delta-D_s).\]
We further have
\[(I-H)^{-1}f(t)= \int_{-\infty}^t e^{s-t} G(t-s)f(s)\,ds,\]
for all $t\in\R$ and $f\in L^1(\R,L^1(\R^N))$, and hence
\[ D(H) \hookrightarrow W^{\rho-\sigma}_1(\R,W^{2\sigma}_1(\R^N)),\]
for $\frac{1}{2}<\sigma<\rho<1$ and the usual Slobodecki\u{\i} spaces.
This fact follows for bounded time intervals from \cite[Theorem 19]{DB84} and \cite[Theorem 4]{DB89}.
The extension to
the time interval $\R$ can be done as in \cite[Chapter 4]{lunardi}.
Let $a<b$ and $R>0$. Due to Sobolev's embedding theorem the space
$W^{2\sigma}_1(B(0,R))$
is embedded into $W^{1}_p(B(0,R))$ for some $p>1$. Hence,
$W^{2\sigma}_1(B(0,R))$ is compactly embedded into
$L^p(B(0,R))\hookrightarrow L^1(B(0,R))$. Corollary~2 in
\cite{Si} now implies that
$\overset{\circ}{W}{}^{\rho-\sigma}_1((a,b),W^{2\sigma}_1(B(0,R)))$
is compactly embedded into $L^1((a,b)\times B(0,R)) $.

\begin{thm}\label{L1-thm}
Assume that $a=I$ and that conditions {\rm (A2)}--{\rm (A5)} are satisfied for some
$\beta,\gamma >0$, $\theta <1$ and $p=1$.  Then, the following assertions hold.
\begin{enumerate}[(a)]
\item
The operator $L_1$ generates a positive and
contractive evolution semigroup $S(\cdot)$ on $L^1(\R^{1+N})$
induced by an evolution family $G(s,r)$,
 $s\ge r$, of positive contractions on $L^1(\R^N)$.
\item
We set $u:= G(\cdot,r)\varphi$ for every  $\varphi\in  L^1(\R^{N})$
and  $r\in \R$. Then, the functions
$(\Delta-D_s)u$ and $Vu$ belong to $L^1((a,b)\times\R^N)$ and
$D_s u(s) = \A(s)u(s)$ for all $s\in (a,b)$ and each interval
$[a,b]\subset  (r,+\infty)$.
Moreover, for each $f\in L^1(\R^{1+N})$ there exists a unique function $u\in \D_1$
satisfying \eqref{parab-eq}, namely
\[u(s)=-L_p^{-1}f(s)=\int_{-\infty}^s G(s,r)f(r)\,dr,\qquad s\in\R.\]
\item
 In addition, assume that condition \eqref{small} holds for some
$p\in (1,+\infty)$.   Then, the evolution semigroups and evolution
 operators obtained in the present
theorem and in Theorem $\ref{Lp-thm}$ coincide on the intersection of the
$L^1$- and $L^p$-spaces.
\end{enumerate}
\end{thm}
\begin{proof}
Take $u \in C_c^\infty(\R^{1+N})$  and set $f=-\cL u$.
We multiply this equation by ${\rm sign}\, u$. Integrating by parts
and using the
dissipativity of $\Delta-D_s$  on $L^1(\R^{1+N})$, we then obtain
\begin{align*}
\int_{\R^{1+N}}&(V+\Div_x F)|u|\,ds\,dx\\
  &\le \int_{\R^{1+N}}(D_s-\Delta) u \, {\rm sign} u\, ds\,dx
       +\int_{\R^{1+N}}(Vu-F\cdot\nabla_x u) \,{\rm sign} u\,ds\,dx\\
  & = \int _{\R^{1+N}}f\,{\rm sign}u\,ds\,dx \le \|f\|_1\,.
\end{align*}
Assumptions (A3) and (A5) thus imply
\begin{align}\label{L1-est}
  (1-\theta)\,\|Wu\|_1 \le \|\cL u\|_1.
\end{align}
Taking into account Proposition~\ref{interpol1} and proceeding as in the
proof of Proposition~\ref{prop-dom-p} after estimate
\eqref{full-apriori}, we obtain the inequalities \eqref{apriori-est}
also for $p=1$ with constants only depending on the constants in (A2)--(A5).
Hence, $L_1$ is closed.  Moreover, the dissipativity of $L_1$  follows from
Lemmas~\ref{dissip} and \ref{density}.

We want to show the invertibility of $I-L_1$. Here,
we may assume that $F\equiv 0$ since the general case is then deduced by means
of the continuity
method as in the proof of Theorem~\ref {Lp-thm}.
We use the notation introduced in that proof. Observe that the operator
$\cL_{0,\eps }=\Delta-V_\eps -D_s$ with domain $D(\Delta-D_s)$
generates a contraction semigroup on $L^1(\R^{1+N})$ for each
$\eps\in(0,1]$, thanks to the bounded perturbation theorem applied
to the generator $\Delta -D_s$. As a consequence,
for each $f\in L^1(\R^{1+N})$ there exists a function $u_\eps \in L^1(\R^{1+N})$
such that $u_\eps- \cL_{0,\eps }u_\eps=f$.  The dissipativity
and \eqref{L1-est} now imply
\begin{equation}\label{approx}
\|u_\eps\|_1 \le \|f\|_1 ,\qquad \|W_\eps u_\eps \|_1 \le c\,\|f\|_1,
\end{equation}
with a constant $c$ independent of $\eps$ since $V_\eps$ and $W_\eps$
satisfy the assumptions (A1)-(A5) with uniform constants. It follows that
\[\|(\Delta-D_s) u_\eps\|_1 =\|\cL_{0,\eps }u_\eps + V_\eps u_\eps\|_1
\le (2+c)\|f\|_1.\]
By the observations made above the statement of the theorem,
there exists a null sequence $(\eps_n)$  such that $u_n:=u_{\eps_n}$
 converges to a function $u$ in $L^1_{\loc}(\R^{1+N})$. We infer from
(\ref{approx}) that $\|u\|_1 \le \|f\|_1$ and  $\|Wu\|_1 \le
c\,\|f\|_1$. Moreover, $(\Delta-D_s) u_n \to (\Delta-D_s) u$ in
$L^1_{\loc}(\R^{1+N})$ and, therefore, $\cL u=f$ and $u \in \D_1$.
Thus, $L_1$ generates a contraction semigroup. The other assertions
can now be shown as in Theorem~\ref{Lp-thm}.
\end{proof}

We now come to the space $C_0$. In the proof of the next result we have
to estimate the oscillation of $V$ itself, and thus we cannot work
with the auxiliary potential $W$.

\begin{prop} \label{prop-dom-0}
 Assume that $a=I$ and that conditions {\rm (A2)} and {\rm (A4)} hold for every
 $\beta,\gamma>0$ and with $W=V$.  Then, there exists a constant $C_\infty>0$
(only depending on  the constants in {\rm (A2)} and {\rm (A4)}) such that
\begin{equation*}
C_\infty^{-1}\,\|u\|_{\D_\infty} \le\|\cL u\|_\infty +\|u\|_\infty
  \le C_\infty\, \|u\|_{\D_\infty},
\end{equation*}
for any $u\in \D_\infty$.
\end{prop}

\begin{proof}
The second estimate in the assertion is a consequence of
Proposition~\ref{interpol1}. Lemma~\ref{density} then
shows that it is enough to prove the other inequality for
all test functions $u$. At first we assume that (A2)
holds with $K_\beta=K'_\gamma=0$ for some $\beta,\gamma\in(0,1]$
to be fixed below.

Let $u\in C_c^\infty (\R^{1+N})$. Set $f=\cL u$.  Again we write
$\nabla$ instead of $\nabla_x$.
Fix $(s_0,x_0) \in \R^{1+N}$.
As in the proof of Proposition~\ref{interpol1} (with $\ell_1=1$
and $\ell_2=\ell$),
we define $Q=Q(s_0,x_0)=(s_0-\tau,s_0+\tau)\times B(x_0,r)$ and $Q'=Q'(s_0,x_0)
=(s_0-\frac{\tau}{2},s_0+\frac{\tau}{2})\times B(x_0,\frac{r}{2})$ with
 $\tau= (4\beta V(s_0,x_0))^{-1}$
 and $r=\sqrt{3}(2\ell\gamma V(s_0,x_0)^\frac12)^{-1}$. Here, we fix
$\ell\ge1$ such that
 \[ \frac{3(2\ell -1)^2}{4(2\ell)^2}\ge \frac23 \qquad \text{and} \qquad
    \frac{5(2\ell +1)^2}{4(2\ell)^2}\le \frac32\,. \]
The  inequalities \eqref{osc} (with $W=V$) now imply  that
\begin{align} \label{osc1}
\frac23\, V(s,x)\le V(s_0,x_0) \le \frac32 \, V(s,x) \quad \text{and} \quad
    | V(s,x) -  V(s_0,x_0)| \le \frac12\, V(s,x),
\end{align}
for all $(s,x)\in Q$.  We choose functions $\eta \in C_c^\infty (\R^N)$ and
$\zeta\in C_c^\infty(\R)$ such that $\1_{B(x_0,r/2)}\le\eta\le\1_{B(x_0,r)}$ and
$\1_{(s_0-\tau/2,s_0+\tau/2)}\le\zeta\le\1_{(s_0-\tau, s_0+\tau)}$, $|\nabla \eta| \leq c/r$,
and $|D^2 \eta| \leq c/r^2$ and $|D_s \zeta| \le c/\tau$. Here and below
the constants $c=c(\eta,\zeta)$ do not depend on
$s_0$, $x_0$, $\tau$  and $r$. We have
\begin{align*}
\Delta &(\zeta\eta u)+F\cdot \nabla (\zeta\eta u)-D_s(\zeta\eta u)
   -V(s_0,x_0) \zeta \eta u\\
&=\zeta \eta f+u (\Delta-D_s)(\zeta\eta)+2\zeta \nabla u \cdot \nabla \eta+
   \zeta u F\cdot \nabla \eta +(V-V(s_0,x_0))\zeta \eta u=:w.
\end{align*}
Since $V(s_0,x_0) >0$, the dissipativity of $\Delta+F\cdot\nabla-D_s$ on
 $C^\infty_c(\R^{1+N})$ (see Lemma~\ref{dissip}) yields
\begin{align*}
\|V(s_0,x_0)\zeta \eta u\|_{\infty,Q} \le \|w\|_\infty \le& \|f\|_{\infty,Q}
   +\Big(\frac{c}{r^2} +\frac{c}{\tau}\Big) \|u\|_{\infty,Q}
 +\frac{c}{r}\|\nabla u\|_{\infty,Q} \\
  &+\frac{c\kappa}{r}\|V^{\frac{1}{2}}u\|_{\infty,Q}
        +\|(V-V(s_0,x_0))u\|_{\infty,Q}\,,
\end{align*}
where we have also used (A4) and
have denoted the sup norm  on $Q$ by $\|\cdot\|_{\infty,Q}$. From  \eqref{osc1}
and the definition of $\tau$ and $r$, we then  deduce
\[
\frac{2}{3}\,\|Vu\|_{\infty, Q'}
  \leq \|f\|_\infty  + c_1(\gamma^2+\kappa\gamma+\beta)\,\|Vu\|_{\infty,Q}
  + \gamma c_1\,\|V^{\frac{1}{2}}\nabla u\|_{\infty,Q}
    + \frac{1}{2}\,\|Vu\|_{\infty,Q},
\]
where $c_1$ only depends on $\eta$ and $\zeta$. Letting $(s_0,x_0)$ vary in $\R^{1+N}$,
we obtain
\[ \frac{2}{3}\,\|Vu\|_\infty
\leq \|f\|_\infty  + c_1(\gamma^2+\kappa\gamma+\beta)\,\|Vu\|_{\infty}
   + \gamma c_1\,\|V^{\frac{1}{2}}\nabla u\|_{\infty}
     + \frac{1}{2}\,\|Vu\|_{\infty}\,.  \]
We fix $\beta=\min\{1,(18 c_1)^{-1}\}$ and take
 $\gamma\le\gamma_0$ where $\gamma_0\in(0,1]$ satisfies
$c_1(\gamma^2_0+\kappa\gamma_0)\le 18^{-1}$. It then follows that
\begin{align} \label{v-est}
\|Vu\|_\infty \leq 18\,\|f\|_\infty +18\gamma_0 c_1\|V^{\frac{1}{2}}\nabla u\|_\infty.
\end{align}
Now (A4) and the equation $\cL u=f$ imply
\begin{align} \label{heat-est-0}
\|(\Delta-D_s) u\|_\infty  \le c_2\,(\|f\|_\infty +\|V^{\frac{1}{2}}\nabla u\|_\infty),
\end{align}
for $c_2:=\max\{19,\kappa+18\gamma_0 c_1\}$. Combining Proposition~\ref{interpol1}
with \eqref{v-est} and \eqref{heat-est-0}, we arrive at
\begin{align*}
\|V^{\frac{1}{2}}\nabla u\|_\infty
&\leq \eps \|(\Delta-D_s) u\|_\infty +\frac{\alpha}{\eps}\|Vu\|_\infty\\
& \le c(\eps) \|f\|_\infty+ \Big(\eps c_2 +\frac{18 \alpha \gamma c_1}{\eps}\Big)
      \|V^{\frac{1}{2}}\nabla u\|_\infty\,,
\end{align*}
for all $\varepsilon\in (0,1]$.
Because of Remark~\ref{alpha-rem}, the constant  $\alpha$ is independent of
$\gamma$ varying in bounded sets.
Finally, we set  $\eps =\gamma^{\frac{1}{2}}$ and  $\gamma =
\min\{\gamma_0, (2(c_2+18\alpha c_1))^{-2}\}$. This leads to the estimate
 $\|V^{\frac{1}{2}}\nabla u\|_\infty \leq c\,\|f\|_\infty$
 for a constant only depending on $N$ and $\kappa$.
 Inequalities \eqref{v-est} and
 \eqref{heat-est-0} now yield  $\|u\|_{\D_\infty} \leq c\,\|\cL u\|_\infty $.

It remains to remove the restriction that $K_\beta=K'_\gamma=0$.
Above we have fixed $\beta,\gamma>0$ depending only on
$N$ and $\kappa$. There exists a number $\lambda=\lambda(\beta,\gamma)\ge0$
such that $V+\lambda$ satisfies (A2) for the fixed value
of $\beta$ and $\gamma$ with $K_\beta=K'_\gamma=0$. Hence, the first
estimate in the assertion holds for $V+\lambda$ and all test functions $u$.
It then holds for $V$ itself with a possibly larger constant $C_\infty$.
\end{proof}

As before Theorem~\ref{L1-thm}, one can verify that $(V(t)f)(s) =G(t)f(s-t)$
defines a positive contraction semigroup on $C_0(\R,C_0(\R^N))$ whose
generator is given by $\Delta-D_s$ on its maximal domain.

\begin{thm}
Assume that $a=I$, that $V\in C(\R^{1+N})$ and that {\rm (A2)}--{\rm (A4)}
 are satisfied  for all $\beta,\gamma >0$. Then, the following assertions hold.
\begin{enumerate}[(a)]
\item
The operator $L_\infty=(\cL,\D_\infty)$ generates a positive and
contractive evolution semigroup $S(\cdot)$ on  $C_0(\R^{1+N})$
induced by an evolution family $G(s,r)$,
 $s\ge r$, of positive contractions on $C_0(\R^N)$.
\item
 We set $u:= G(\cdot,r)\varphi$ for every  $\varphi\in  C_0(\R^{N})$
and $r\in \R$. Then, the functions
$(\Delta-D_s)u$ and $Vu$ belong to $C([a,b], C_0(\R^N))$ and
$D_s u(s) = \A(s)u(s)$ for all $s\in (a,b)$ and each interval
$[a,b]\subset  (r,+\infty)$.
Moreover, for each $f\in C_0(\R^{1+N})$ there exists a unique function $u\in \D_\infty$
satisfying \eqref{parab-eq}, namely
\[u(s)=-L_p^{-1}f(s)=\int_{-\infty}^s G(s,r)f(r)\,dr,\qquad s\in\R.\]
\item
If also the assumptions of Theorems~$\ref{Lp-thm}$ or $\ref{L1-thm}$ hold
for some $p\in [1,+\infty)$, then the evolution  semigroup and the evolution family  obtained in the present
theorem and in Theorems $\ref{Lp-thm}$ or $\ref{L1-thm}$, respectively,
coincide on the intersection of the $C_0$- and $L^p$-spaces.
\end{enumerate}
\end{thm}

\begin{proof}
We first show that $L_\infty$ generates a contraction semigroup on
$C_0(\R^{1+N})$ in the case when $V=W$.
Lemmas~\ref{dissip} and \ref{density} and Proposition~\ref{prop-dom-0}
show that $L_\infty$ is closed, densely defined and dissipative.
Moreover, as in the proof of Theorem~\ref{Lp-thm},
we can restrict ourselves to the case $F\equiv 0$ since
Proposition~\ref{prop-dom-0} gives a suitable a priori estimate.
We use the notation introduced in that proof. Replacing $V$ by $V+\lambda$
we can suppose that $K_\beta=K'_\gamma=0$. We fix $p>N+2$
and sufficiently small $\beta,\gamma>0$ such that \eqref{small}
hold for this $p$, $M=1$ and $\theta=\kappa=0$. Observe that
$W^{1,2}_p(\R^{1+N})\hookrightarrow C_0(\R, C^1_0(\R^N))$.
For each $f\in C_c(\R^{1+N})$ and each $\varepsilon\in (0,1]$,
there exists a function $u_\eps \in W^{1,2}_p(\R^{1+N})$  such that
\begin{align}\label{eps-eq}
u_\eps- (\Delta -V_\eps -D_s)u_\eps =f,
\end{align}
since $V_\eps$ is a bounded perturbation of the generator $\Delta -D_s$.
By dissipativity, we have $\|u_\eps\|_r \le \|f\|_r$ for $r=p,+\infty$.
Propositions~\ref{prop-dom-p} and \ref{prop-dom-0} also yield
\[
\|(\Delta-D_s)u_\eps\|_r+\|V_\eps u_\eps \|_r \le c\,\|f\|_r,
\]
for $r=p,+\infty$.
Here and below the constants $c$ do not depend on $\eps$.
Since the sequence $(u_{\eps})$ is bounded in $W^{1,2}_p(\R^{1+N})$, which continuously
embeds in $C^{\alpha}(\R^{1+N})$ for a suitable $\alpha>0$, the
Arzel\`a-Ascoli theorem implies that
$u_{\eps_n}$ converges locally uniformly in $\R^{1+N}$ to a function $u$,
for a suitable null sequence $(\eps_n)$.
Due to \eqref{eps-eq}, also $(\Delta-D_s) u_{\eps_n}$ converges uniformly
 on compact sets so that
$$u-(\Delta-D_s) u +Vu=f \quad \text{ and  }\quad
 \|Vu\|_r+\|(\Delta-D_s) u\|_r \le c\, \|f\|_r,$$
 for $r=p, +\infty$. Therefore, $u\in W^{1,2}_p(\R^{1+N})\hookrightarrow
  C_0(\R, C^1_0(\R^N))$. We next  show that $Vu$ belongs to
$C_0(\R^{1+N})$. Take $(s_0,x_0)\in \R^{1+N}$ and
$\eta \in C^\infty_c (\R^{1+N})$ such that $\1_{B((s_0,x_0),R)}\le\eta\le\1_{B((s_0,x_0),2R)}$,
$\|\nabla \eta\|_{\infty} \le c/R$ and $\|D^2\eta\|_{\infty} \le c/R^2$ for all $R\ge1$. Then
$$
\eta u -(\Delta-D_s) (\eta u)+V\eta u=\eta f
     -2 \nabla_x \eta \cdot \nabla_x u -u \Delta \eta +uD_s\eta
$$
and Proposition~\ref{prop-dom-0} (applied to $\eta u\in \D_\infty$) shows that
\begin{align*}
|V(s_0,x_0)&u(s_0,x_0)|\le \|V\eta u\|_\infty\\
& \le c \Bigl [\|\eta f\|_\infty
    +\|\eta u\|_\infty +\frac{1}{R}\|\nabla_x u\|_\infty+
     \frac{1}{R}\|u\|_\infty \Bigr ]\\
&\le c \Bigl [\|f\|_{L^{\infty}(B((s_0,x_0),R))}
    +\|u\|_{L^{\infty}(B((s_0,x_0),R))} +\frac{1}{R}\|\nabla_x u\|_\infty+
     \frac{1}{R}\|u\|_\infty \Bigr ].
\end{align*}
Fix $\varepsilon>0$ and let $R$ be sufficiently large such that $R^{-1}(\|u\|_{\infty}+\|\nabla_xu\|_{\infty})\le\varepsilon$.
Further, let $M>R$ be so large such that $|f(s,x)|+|u(s,x)|\le\varepsilon$ for any $|(s,x)|\ge M$ (this is possible
since $u,f\in C_0(\R^{1+N})$). The above inequality implies that
$|Vu(s_0,x_0)u(s_0,x_0)|\le 2c\varepsilon$ if $|(s_0,x_0)|\ge M+R$, so that
$Vu \in C_0(\R^{1+N})$.
We conclude that
 $(\Delta-D_s) u=u+Vu-f\in C_0(\R^{1+N})$ and $u\in \D_\infty$. As a
consequence, $I-L_\infty$ has dense range and thus $L_\infty$
(also with $F\neq0$) generates a
contraction semigroup, provided that $V=W$. Given $0\le f\in  C_0(\R^{1+N})$
and $\lambda>0$, there is a function
$u\in \D_\infty$ with $u-L_\infty u=f$. If $u$ were not non-negative,
it would have a strictly negative minimum. This fact would
easily lead to a contradiction. Hence, $L_\infty$ has a positive resolvent
and generates a positive semigroup.

Now, let $V$ be as in the statement. For $0\le \tau\le 1$, we introduce the
potential $V_\tau =W+\tau(V-W)$ and the operator
 $\cL_\tau=\Delta +F \cdot \nabla -V_\tau-D_s$
with $D(\cL_\tau)= \D_\infty$. Observe that $V_0=W
\le V_\tau\le V_1=V$. We know that $\cL_0$ generates a positive contraction
semigroup on $C_0(\R^{1+N})$. Whenever also
$\cL_\tau$ generates a contraction  semigroup $(e^{t\cL_\tau})_{t\ge0}$,
we can apply the Lie-Trotter formula to the sum $\cL_\tau=\cL_0+ W-V_\tau$
to derive that $0\le e^{t\cL_\tau}\le e^{t\cL_0}$ for
all $t\ge0$, see \cite[Corollary~III.5.8]{EN}. Since $(I-\cL_{\tau})^{-1}$ and
$(I-\cL_0)^{-1}$ are, respectively, the Laplace transform of $e^{t\cL_\tau}$ and $e^{t\cL_0}$ at $\lambda=1$, we  obtain
 $0 \le (I-\cL_\tau)^{-1} \le(I-\cL_0)^{-1}$ and, using also (A3),
$$0\le (V_\sigma-V_\tau )(I-\cL_\tau)^{-1}
\le (c_1-1)(\sigma-\tau)W(I-\cL_0)^{-1}$$
for all $0\le \tau\le1$ such that $\cL_\tau$
generates a positive contraction semigroup and for all $\sigma\in [\tau,1]$.
On the other hand, Proposition~\ref{prop-dom-0} implies that
$\|W(I-\cL_0)^{-1}\| \le 3C_\infty$. By finitely many perturbation steps
of the form $\cL_\sigma=\cL_\tau +V_\tau-V_\sigma$,
we can then  conclude that $\cL_1=L_\infty$ generates a
positive contraction semigroup
on $C_0(\R^{1+N})$.
The remaining assertions can be shown as in Theorem~\ref{Lp-thm}.
\end{proof}

\appendix
\section{A variant of the Besicovitch covering theorem}

In this appendix, we prove a variant of the classical Besicovitch
covering theorems, in which balls are replaced by cylinders. This
proposition plays a crucial role in the proof of Proposition
\ref{interpol1}.

Let us introduce the distance $d$ on $\R^{1+N}$ defined by
\begin{equation}
d((t,x),(s,y))=\max\{|t-s|^{1/2},|x-y|\},\qquad\;\,(t,x),\,(s,y)\in\R^{1+N}.
\label{metric}
\end{equation}

A straightforward computation shows that $d$ is in fact a metric
which defines the same topology in $\R^{1+N}$ as
 the Euclidean norm $|\,\cdot\,|$. Moreover,
$(\R^{1+N},d)$ and $(\R^{1+N},|\cdot|)$ have the same bounded
sets. For all $(s_0,x_0)\in\R^{1+N}$ and  $r>0$, we denote by
$B_d((s_0,x_0),r)$ the ball with center at $(s_0,x_0)$ and radius
$r$ in the metric $d$. Note that
\begin{equation}
B_d((s_0,x_0),r)=(s_0-r^2,s_0+r^2)\times B(x_0,r).
\label{cylinder}
\end{equation}

We can now state and prove the following proposition.

\begin{prop}
\label{prop-besicovitch} Let $\varrho:\R^{1+N}\to (0,+\infty)$ be a bounded
Lipschitz continuous function (with respect to the distance $d$)
with Lipschitz constant $\kappa<1$, i.e.,
\begin{eqnarray*}
|\varrho(s,x)-\varrho(r,y)|\le \kappa
d((s,x),(r,y)),\qquad\;\,(s,x),\,(r,y)\in\R^{1+N}.
\end{eqnarray*}
 Then, there exists a sequence
$((s_n,x_n))\subset\R^{1+N}$ such that the family ${\mathscr
F}=\{B_d((s_n,x_n);\varrho(s_n,x_n)): n\in\N\}$ is a covering of
$\R^{1+N}$. Moreover, for each $\lambda\in [1,\kappa^{-1})$ there exists
a number $\zeta(\kappa,\lambda,N)$ such that every subset $I\subset\N$
with $\bigcap_{n\in I}B_d((s_n,x_n),\lambda\varrho(s_n,x_n))\neq\varnothing$
contains at most $\zeta(\kappa,\lambda,N)$ elements.
\end{prop}

\begin{proof}
We adapt partly the proof of the classical Besicovitch covering
theorem given in \cite[Section 1.5.2, Theorem 2]{evans-gariepy} to our
situation. Being rather long, we split the proof into several steps.

{\em Step 1.} Let us set
\begin{eqnarray*}
\delta=\sup\{\varrho(s,x):~(s,x)\in\R^{1+N}\},
\end{eqnarray*}
and define the sets
\begin{align*}
A^{(l)}&=\{(s,x)\in\R^{1+N}:~\omega (l-1)\le d((s,x),(0,0))\le \omega l\}\\
\delta^{(l)}_1&=\max\{\varrho(s,x): (s,x)\in
A^{(l)}\},\quad\;\,l\in\N,
\end{align*}
where $\omega$ is a positive constant greater than $2\kappa^{-1}\delta$.
For each $l\in\N$, we are going to construct a countable family of
balls ${\mathcal
F}^{(l)}=\{B_d((s^{(l)}_n,x^{(l)}_n),\varrho(s^{(l)}_n,x^{(l)}_n)):
n\in\N\}$, which, as we will show in the forthcoming steps, will
represent a countable covering of the set $A^{(l)}$. The family
${\mathscr F}$ we are looking for will be then defined as the union
of all the balls from the families ${\mathscr F}^{(l)}$ ($l\in\N$).

We set $A^{(l)}_1:=A^{(l)}$.
Let us fix $l\in\N$ and an arbitrary point
$(s_1^{(l)},x_1^{(l)})\in A^{(l)}_1$ such that
$\varrho(s_1^{(l)},x_1^{(l)})\ge \frac{3}{4}\delta^{(l)}_1$. Next,
we consider the set $A_2^{(l)}:=A^{(l)}_1\setminus
B_d((s_1^{(l)},x_1^{(l)}),\varrho(s_1^{(l)},x_1^{(l)}))$, set
$\delta_2^{(l)}:=\max\{\varrho(s,x): (s,x)\in A_2^{(l)}\}$, and we
pick up an arbitrary point $(s_2^{(l)},x_2^{(l)})\in A_2^{(l)}$ such
that $\varrho(s_2^{(l)},x_2^{(l)})\ge \frac{3}{4}\delta_2^{(l)}$. We
then inductively define the sequence $(s_n^{(l)},x_n^{(l)})$ in this
way: $(s_m^{(l)},x_m^{(l)})$ is any arbitrary fixed point in
$A_m^{(l)}:=A_{m-1}^{(l)}\setminus
B_d((s^{(l)}_{m-1},x^{(l)}_{m-1}),\varrho(s^{(l)}_{m-1},x^{(l)}_{m-1}))$
such that $\varrho(s_m^{(l)},x_m^{(l)})\ge 3\delta_m^{(l)}/4$, where
$\delta_m^{(l)}=\max\{\varrho(s,x): (s,x)\in A_m^{(l)}\}$.

We have two possibilities: either there exists $m\in \N$ such that
$A_{m+1}^{(l)}=\varnothing$ or $A_n\neq\varnothing$ for all
$n\in\N$. In the first case, we set $I^{(l)}=\{1,\ldots,m_0^{(l)}\}$,
where $m_0^{(l)}$ is the smallest integer such that
$A^{(l)}_{m_0^{(l)}+1}=\varnothing$. In the second case, we set
$I^{(l)}=\N$.

Let $\lambda>0$. In the sequel, to simplify the notation, we set
\begin{equation}
B_{i,\lambda}^{(1)}:=B_d((s_i^{(l)},x_i^{(l)}),\lambda\varrho(s_i^{(l)},x_i^{(l)})),
\quad\;\,B^{(l)}_i:=B^{(l)}_{i,1},\quad\;\,
\varrho^{(l)}_i:=\varrho(s_i^{(l)},x_i^{(l)}). \label{notation}
\end{equation}

{\em Step 2.} Here, for every $l\in\N$, we prove that the balls
$B_{i,1/3}^{(l)}$ ($i\in I^{(l)}$) are all disjoint. For this
purpose we first observe that $\varrho^{(l)}_i\ge
\frac{3}{4}\varrho^{(l)}_j$ if $j>i$ . Indeed,
\begin{align*}
\varrho_i^{(l)}\ge \frac{3}{4}\max\{\varrho(s,x): (s,x)\in
A_i^{(l)}\}\ge\frac{3}{4}\max\{\varrho(s,x): (s,x)\in
A_j^{(l)}\}\ge\frac{3}{4}\varrho^{(l)}_j,
\end{align*}
since $A_i\supset A_j$.

Using this inequality, we can now prove that the balls
$B_{i,1/3}^{(l)}$ ($i\in I^{(l)}$) are all disjoint. Assume, by
contradiction, that there exists $(s,y)\in B^{(l)}_{i,1/3}\cap
B^{(l)}_{j,1/3}$ for some indexes $i$ and $j$. Then, the triangle
inequality yields
\begin{align*}
d((s_i^{(l)},x_i^{(l)}),(s_j^{(l)},x_j^{(l)}))&\le
d((s_i^{(l)},x_i^{(l)}),(s,y)) +d((s,y),(s_j^{(l)},x_j^{(l)}))\notag\\
& \le \frac{1}{3}\varrho_i^{(l)}+\frac{1}{3}\varrho_j^{(l)}\le
\frac{1}{3}\varrho_i^{(l)}+\frac{4}{9}\varrho_i^{(l)}
=\frac{7}{9}\varrho_i^{(l)}.
\end{align*}
As a result, $(s_j^{(l)},x_j^{(l)})\in B_i^{(l)}$. This is impossible
since, by construction, the point
$(s_j^{(l)},x_j^{(l)})$ belongs to the set $A_j^{(l)}$ which is
contained in the complement of  $B_i^{(l)}$.

{\em Step 3.} Here, we show for the case $I^{(l)}=\N$ that the sequence
$(\varrho_{n}^{(l)})$ tends to 0 as $n\to +\infty$. As we have already
noticed, $(s_m^{(l)},x_m^{(l)})\notin B_n^{(l)}$ if $m>n$. Hence,
\begin{equation}
d((s_m^{(l)},x_m^{(l)}),(s_n^{(l)},x_n^{(l)}))\ge \varrho^{(l)}_n
=\frac{1}{3}\varrho^{(l)}_n+\frac{2}{3}\varrho^{(l)}_n\ge
\frac{1}{3}\varrho^{(l)}_n+\frac{1}{2}\varrho^{(l)}_m \ge
\frac{1}{3} \left (\varrho^{(l)}_n+\varrho^{(l)}_m\right ).
\label{evans-1}
\end{equation}
Since $(s_n^{(l)},x_n^{(l)})\in A^{(l)}$ for any $n\in\N$, the
sequence $((s_n^{(l)},x_n^{(l)}))$ is bounded with respect to the
distance $d$ and, by the remarks at the very beginning of the
section, it is bounded with respect to the Euclidean norm as well.
Thus, there exists a subsequence $(t_{n_k}^{(l)},x_{n_k}^{(l)})$ which
converges with respect to the Euclidean norm (and, hence, with
respect to the distance $d$) to a point $(s,x)\in A_1^{(l)}$. From
\eqref{evans-1}, it follows that the sequence
$(\varrho^{(l)}_{n_k})$ tends to $0$ as $k\to +\infty$. The same
arguments can then be used to prove that any subsequence of
$(\varrho^{(l)}_n)$ has a subsequence which converges to $0$.
Hence, $\varrho^{(l)}_n$ tends to $0$ as $n\to +\infty$, as well.

{\em Step 4.} We can now prove that, for each $l\in\N$, the family
${\mathscr F}^{(l)}$ is a covering of the set $A^{(l)}$. Of course,
we have only to consider the case when $I^{(l)}=\N$. So, let us fix
a point $(s^*,x^*)\in A^{(l)}$. Since, by Step 3, the sequence
$(\varrho_n^{(l)})$ vanishes as $n\to +\infty$, we can fix $n_0\ge
2$ such that $\varrho_{n_0}^{(l)}<3\varrho(s^*,x^*)/4$. This implies
that $(s^*,x^*)\in B_j^{(l)}$ for some $j\le n_0-1$. Indeed, if this
were not the case, then $(s^*,x^*)\in A^{(l)}_{n_0}$ and, hence,
\begin{eqnarray*}
\varrho_{n_0}^{(l)}\ge \frac{3}{4}\max\{\varrho(s,x): (s,x)\in
A_{n_0}\}\ge \frac{3}{4}\varrho(s^*,x^*),
\end{eqnarray*}
a contradiction.

{\em Step 5.} Here, we prove that, for every $l\in\N$ and every $\lambda\in
[1,\kappa^{-1})$, there exists $\xi(\kappa,\lambda,N)$ such that any
ball of the family ${\mathscr F}^{(l)}_{\lambda}:=\{B_{i,\lambda}^{(l)}:
i\in I^{(l)}\}$ intersects at most $\xi(\kappa,\lambda,N)$ other balls
of the family. Here, $B_{i,\lambda}^{(l)}$ is defined by
\eqref{notation}. As a byproduct, we then deduce that, if $J\subset
I^{(l)}$ is a finite set of indexes such that $\bigcap_{i\in
J}B_{i,\lambda}^{(l)}\neq\varnothing$, then $J$ contains at most
$\xi(\kappa,\lambda,N)$ elements.

Let us fix a ball $B_{i_0,\lambda}^{(l)}$ and let $J$ be a finite set
of indexes such that $B_{i,\lambda}^{(l)}\cap
B_{i_0,\lambda}^{(l)}\neq\varnothing$ for every $i\in J$. Clearly,
\begin{align*}
d((s_{i_0}^{(l)},x_{i_0}^{(l)}),(s_i^{(l)},x_i^{(l)})) \le \lambda\left
(\varrho_{i_0}^{(l)}+ \varrho_i^{(l)}\right ).
\end{align*}
Since, by assumptions the function $\varrho$ is $\kappa$-Lipschitz
continuous, we have
\begin{eqnarray*}
|\varrho_{i_0}^{(l)}-\varrho_i^{(l)}| \le \kappa
d((s_{i_0}^{(l)},x_{i_0}^{(l)}),(s_i^{(l)},x_i^{(l)})).
\end{eqnarray*}
Hence,
\begin{eqnarray*}
|\varrho_{i_0}^{(l)}-\varrho_i^{(l)}| \le \kappa\lambda\left
(\varrho_{i_0}^{(l)}+ \varrho_i^{(l)}\right )
\end{eqnarray*}
or, equivalently,
\begin{equation}
\varrho_{i_0}^{(l)}\le\frac{\kappa\lambda+1}{1-\kappa\lambda}\varrho_i^{(l)},\qquad
\varrho_i^{(l)}\le\frac{\kappa\lambda+1}{1-\kappa\lambda}\varrho_{i_0}^{(l)}.
\label{cover-1}
\end{equation}
We now observe that for all $i\in J$ and $(s,x)\in
B_{i,1/3}^{(l)}$ it holds that
\begin{align*}
d((s,x),(s^{(l)}_{i_0},x^{(l)}_{i_0}))&\le
d((s,x),(s^{(l)}_i,x^{(l)}_i))+
d((s^{(l)}_i,x^{(l)}_i),(s^{(l)}_{i_0},x^{(l)}_{i_0}))\\
& \le \frac{1}{3}\varrho^{(l)}_i +\lambda\left (\varrho_{i_0}^{(l)} +
\varrho_i^{(l)}\right ) =\left (\frac{1}{3}+\lambda\right
)\varrho^{(l)}_i
+\lambda\varrho_{i_0}^{(l)}\\
& \le \left \{\left (\frac{1}{3}+\lambda\right )
\frac{\kappa\lambda+1}{1-\kappa\lambda}+\lambda\right\}\varrho_{i_0}^{(l)}=
\frac{\kappa\lambda+6\lambda+1}{3-3\kappa\lambda}\varrho_{i_0}^{(l)}.
\end{align*}
Therefore, $B^{(l)}_{i,1/3}\subset B^{(l)}_{i_0,\sigma_\kappa}$ for
every $i\in J$, where
$\sigma_{\kappa}:=(\kappa\lambda+6\lambda+1)/(3-3\kappa\lambda)$. Now,
recalling that the balls $B^{(l)}_{i,1/3}$ ($i\in I^{(l)}$) are all
disjoint, it follows that
\begin{align}
3^{-N-2}2\omega_N\sum_{i\in J}(\varrho_i^{(l)})^{N+2}&=m\left
(\bigcup_{i\in J}B^{(l)}_{i,1/3}\right )\notag\\
&\le m(B^{(l)}_{i_0,\sigma_\kappa})=2\omega_N\left
(\frac{\kappa\lambda+6\lambda+1}{3-3\kappa\lambda}\right
)^{N+2}(\varrho_{i_0}^{(l)})^{N+2}, \label{cover-2}
\end{align}
where $m$ and $\omega_N$ denote, respectively, the Lebesgue measure
in $\R^N$ and the Lebesgue measure of the unit ball in $\R^N$. Using
\eqref{cover-1} we can estimate
\begin{equation}
\sum_{i\in J}(\varrho_i^{(l)})^{N+2} \ge {\rm card}(J)\left
(\frac{1-\kappa\lambda}{\kappa\lambda+1}\right
)^{N+2}(\varrho_{i_0}^{(l)})^{N+2}. \label{cover-3}
\end{equation}
 From \eqref{cover-2} and \eqref{cover-3} we now get
\begin{eqnarray*}
3^{-N-2}{\rm card}(J)\left
(\frac{1-\kappa\lambda}{\kappa\lambda+1}\right
)^{N+2}(\varrho_{i_0}^{(l)})^{N+2}\le \left
(\frac{\kappa\lambda+6\lambda+1}{3-3\kappa\lambda}\right
)^{N+2}(\varrho_{i_0}^{(l)})^{N+2},
\end{eqnarray*}
i.e.,
\begin{eqnarray*}
{\rm card}(J)\le \xi(\kappa,\lambda,N):=\left [\left
(\frac{\kappa^2\lambda^2+2\kappa\lambda(1+3\lambda)+6\lambda+1}{(1-\kappa\lambda)^2}\right
)^{N+2}\right ],
\end{eqnarray*}
where $[\,\cdot\,]$ denotes the integer part of the quantity in
brackets.


{\em Step 6.} We now prove that, for every $l\in\N$ and every
$\lambda\in [1,\kappa^{-1})$, the intersection of more than
$\zeta(\kappa,\lambda,N):=2\xi(\lambda,\kappa,N)+2$ balls from the family
${\mathscr F}_{\lambda}:=\{B^{(l)}_{i,\lambda}: l\in\N,\, i\in I^{(l)}\}$
is empty. For this purpose, we reorder each family ${\mathscr
F}^{(l)}_{\lambda}:=\{B^{(l)}_{i,\lambda}: i\in I^{(l)}\}$ ($l\in\N$) into
the union of $\xi(\kappa,\lambda,N)+1$ subfamilies of disjoint balls.
Let us fix $l\in\N$ and define the function
$\sigma^{(l)}_{\lambda}:\N\to\{1,\ldots,\xi(\kappa,\lambda,N)+1\}$
inductively as follows. For $j=1,\ldots,\xi(\kappa,\lambda,N)+1$, we
set $\sigma^{(l)}_{\lambda}(j)=j$. Take an integer $m>  \xi(\kappa,\lambda,N)+1$.
Suppose $\sigma^{(l)}_{\lambda}(j)$ is
defined for every $j\in\{1,\ldots,m\}$. Let us define
$\sigma^{(l)}_{\lambda}(m+1)$. For this purpose, we introduce the set
${\mathcal H}^{(l)}_{\lambda,m}=\{j=1,\ldots,m: B_{j,\lambda}^{(l)}\cap
B_{m+1,\lambda}^{(l)}\neq\varnothing\}$. By Step 5, ${\mathcal
H}^{(l)}_{\lambda,m}$ has less than $\xi(\kappa,\lambda,N)+1$ elements.
Hence, there exists the minimal $h_m\in \{1,\ldots,\xi(\kappa,\lambda,N)+1\}$  such that $h_m\notin \sigma^{(l)}_{\lambda}({\mathcal H}^{(l)}_{\lambda,m})$.
Then we have $B^{(l)}_{r,\lambda}\cap B^{(l)}_{m+1,\lambda}=\varnothing$
for all  $r\in\{1,\dots,m\}$ satisfying $\sigma^{(l)}_{\lambda}(r)=h_m$.
We define $\sigma^{(l)}_{\lambda}(m+1):=h_m$.

Let us now set ${\mathscr G}_{h,\lambda}^{(l)}:=\{B_{i,\lambda}^{(l)}:
\sigma^{(l)}_{\lambda}(i)=h\}$ for each $h\in \{1,\ldots,\xi(\kappa,\lambda,N)+1\}$.
 From the very
definition of the function $\sigma^{(l)}_{\lambda}$, the set ${\mathscr
G}_{h,\lambda}^{(l)}$ consists of disjoint balls. Clearly, each ball of
the family ${\mathscr F}^{(l)}_{\lambda}$ belongs to ${\mathscr
G}^{(l)}_{h,\lambda}$ for a (unique) $h\in\N$. So we have split the
family ${\mathscr F}_{\lambda}^{(l)}$ into the union of the families
${\mathscr G}^{(l)}_{j,\lambda}$ ($j=1,\ldots,\xi(\kappa,\lambda,N)+1$).

We now introduce the sets ${\mathscr G}_{j,\lambda}$
($j=1,\ldots,\zeta(\kappa,\lambda,N)$) defined as follows:
\begin{align*}
&{\mathscr G}_{j,\lambda}=\bigcup_{l=1}^{+\infty}{\mathscr
G}^{(2l-1)}_{j,\lambda},
&j=1,\ldots,\xi(\kappa,\lambda,N)+1,\\
&{\mathscr G}_{j,\lambda}=\bigcup_{l=1}^{+\infty}{\mathscr
G}^{(2l)}_{j-\xi(\kappa,\lambda,N),\lambda},&
j=\xi(\kappa,\lambda,N)+2,\ldots \zeta(\kappa,\lambda,N).
\end{align*}
Note that every family ${\mathscr G}_{j,\lambda}$ consists of disjoint
balls. Indeed, suppose that $B_1$ and $B_2$ belong to ${\mathscr
G}_{j,\lambda}$ for some $j$ and $B_1\cap B_2\neq\varnothing$. (We
assume that $j\le \xi(\kappa,\lambda,N)+1$ but the same argument can be
applied in the case when $j>\xi(\kappa,\lambda,N)+1$.) Then, $B_1\in
{\mathscr G}^{(2l_1-1)}_{j,\lambda}$ and $B_2\in {\mathscr
G}^{(2l_2-1)}_{j,\lambda}$ for some $l_1,l_2\in\N$. Clearly, from the
above results $l_1\neq l_2$ and, without loss of generality, we can
assume that $l_1<l_2$. Denote by $(s_1,x_1)$ and $(s_2,x_2)$ the
centers of the balls $B_1$ and $B_2$, respectively. Since $B_1\cap
B_2\neq\varnothing$, we have
\begin{eqnarray*}
d((s_1,x_1),(s_2,x_2))\le \lambda\left (\varrho(s_1,x_1)+\varrho(s_2,x_2)\right )\le
\lambda\left (\delta+\delta\right )=2\lambda\delta.
\end{eqnarray*}
On the other hand, $(s_1,x_1)\in A^{(2l_1-1)}$ and $(s_2,x_2)\in
A^{(2l_2-1)}$. Hence,
\begin{align*}
d((s_1,x_1),(s_2,x_2))&\ge d((s_2,x_2),(0,0))-d((s_1,x_1),(0,0))\\
& \ge \omega(2l_2-2)-\omega(2l_1-1)\\
& =\omega(2(l_2-l_1)-1) \ge \omega,
\end{align*}
which leads us to a contradiction, since $\omega>2\kappa^{-1}\delta$. It is now clear that
\begin{eqnarray*}
\sum_{j=1}^{\zeta(\kappa,\lambda,N)}\sum_{B_{i,\lambda}^{(l)}\in
{\mathscr G}_{j,\lambda}}\chi_{B_{i,\lambda}^{(l)}}(s,x)\le
\zeta(\kappa,\lambda,N),\qquad\;\,(s,x)\in\R^{1+N},
\end{eqnarray*}
and this completes the proof.
\end{proof}

\end{document}